\documentclass[final,leqno]{siamltex}
\usepackage{mathrsfs}
\usepackage{epsfig}
\usepackage{subfigure}
\usepackage{graphicx}
\usepackage{amsmath,amsfonts,amssymb}
\usepackage{bm}
\makeatletter
\newcommand*{\rom}[1]{\expandafter\@slowromancap\romannumeral #1@}
\makeatother
\usepackage[top=1in, bottom=1in, left=1.25in, right=1.25in]{geometry}

\newcommand{\N}{\ensuremath{\mathbb{N}}}
\newcommand{\Nd}{\ensuremath{\mathbb{N}_0^d}}   

\newcommand{\R}{\ensuremath{\mathbb{R}}}   
\newcommand{\Rd}{\ensuremath{\mathbb{R}^d}}
\renewcommand{\H}{\ensuremath{\mathcal{H}}}
\newcommand{\bH}{\ensuremath{\bm{\mathcal{H}}}} 
\newcommand{\D}{\ensuremath{\mathcal{D}}}
\newcommand{\bD}{\ensuremath{\bm{\mathcal{D}}}}  
\newcommand{\bO}{\ensuremath{\bm{\Omega}}}
\newcommand{\W}{\ensuremath{\mathcal{W}}}
\newcommand{\K}{\ensuremath{\mathcal{K}}}
\newcommand{\cN}{\ensuremath{\mathcal{N}}}
\newcommand{\A}{\ensuremath{\mathcal{A}}}

\def \a {\alpha}
\def \b {\beta}
\def \g {\gamma}
\def \p {\partial}

\def \ba {\bm \alpha}
\def \bb {\bm \beta}
\def \bg {\bm \gamma}
\def \bd {\bm \delta}
\def \bn {\bm n}
\def \bk {\bm k}
\def \bi {\bm i}
\def \bl {\bm l}
\def \br {\bm r}
\def \bp {\bm \partial}
\def \bx {\bm x}

\title{Hermite Spectral Method with Hyperbolic Cross Approximations to High-dimensional Parabolic PDEs\thanks{Received by the editors ***; accepted for publication (in revised form) ***; published electronically ***.}}

\author{ Xue Luo\thanks{Department of Electronic Engineering, Tsinghua University,
Beijing, P. R. China, 100084 and School of Mathematics and Systems Science, Beijing University of Aeronautics and Astronautics (Beihang University), Beijing, P. R. China, 100083 ({\tt xluo6@uic.edu, luoxue0327@163.com}).}
\and Stephen S.-T. Yau\thanks{Department of Mathematical Sciences, Tsinghua
University, Beijing, P. R. China, 100084 ({\tt yau@uic.edu})}. The work of this author was supported by start-up fund from Tsinghua University.}

\begin{document}
\renewcommand{\theequation}{\arabic{section}.\arabic{equation}}

\subtitle{Dedicated to Professor Peter Caines on the occasion of his 68th birthday}

\maketitle

\newtheorem{remark}{Remark}
\renewcommand{\theremark}{\arabic{section}.\arabic{remark}}

\begin{abstract}
	It is well-known that sparse grid algorithm has been widely accepted as an efficient tool to overcome the ``curse of dimensionality" in some degree. In this note, we first give the error estimate of hyperbolic cross (HC) approximations with generalized Hermite functions. The exponential convergence in both regular and optimized hyperbolic cross approximations has been shown. Moreover, the error estimate of Hermite spectral method to high-dimensional linear parabolic PDEs with HC approximations has been investigated in the properly weighted Korobov spaces. The numerical result verifies the exponential convergence of this approach.
\end{abstract}



\pagestyle{myheadings}
\thispagestyle{plain}
\markboth{X. LUO AND S. S.-T. YAU}{HSM WITH HC APPROXIMATION TO PARABOLIC PDES}

\section{Introduction}

\setcounter{equation}{0}

Our study is motivated by solving the conditional density function of the states of certain nonlinear filtering. The conditional density function satisfies a linear parabolic PDE, which comes from the robust Duncan-Mortensen-Zakai equation after some exponential transformation, see \cite{LuoY}, \cite{YY}. We need to solve this equation in $\mathbb{R}^d$, since the states lived in the whole space, where $d$ is the number of the states. Moreover, the real-time solution is expected in the filtering problems, so it is natural to adopt the spectral methods. Among the existing literature, the Hermite and Laguerre spectral methods are the commonly used approaches based on orthogonal polynomials in infinite interval, referring to \cite{FK}, \cite{XW}. Although the Hermite spectral method (HSM) appears to be a natural choice, it is not commonly used as Chebyshev and Fourier spectral method, due to its poor resolution (see \cite{GO}) and the lack of fast algorithm for the transformation (see \cite{B2001}). However, it is shown in \cite{B1980} that an appropriately chosen scaling factor could greatly improve the resolution. Some further investigations on the scaling factor can be found in \cite{T} and also in Chapter 7, \cite{STW}. Moreover, recently a guideline of choosing the suitable scaling factors for Gaussian/super-Gaussian functions is described in \cite{LY}, as well as the application of HSM to 1-dim forward Kolmogorov equation. 

Nevertheless, the number of the states is generally greater than one. Taking the target tracking problem in 3-dim as an example, there are at least six states involved in this system (three for position, three for velocity). That is, we need to solve a linear parabolic PDE in $\mathbb{R}^6$. Naively, if we implement the spectral method with tensor product formulation and assume the first $N$ modes need to be computed in each direction, then the total amount of the computation is $N^6$. Even if with moderately small $N$, it is still not within the reasonable computing capacity. This is the so-called ``curse of dimensionality". An efficient tool to reduce this effect is the sparse grids approximations from Smolyak's algorithm \cite{S}, which is based on a hierarchy of one-dimensional quadrature. It has a potential to obtain higher rates of convergence than many existing methods, under certain regularity conditions. For example, the convergence rate of Monte Carlo simulations are $\mathcal{O}(N^{-\frac12})$ with $N$ sample points, while the sparse grids from \cite{S} achieves $\mathcal{O}(N^{-r}(\log{N})^{(d-1)(r+1)})$, under the condition that the function has bounded mix derivatives of order $r$. The studies of sparse grids start from the basis functions in the physical spaces: piecewise linear multiscale bases \cite{BG2004}, wavelets \cite{BG2004}, \cite{SS}. In the recent one decade, the hyperbolic cross (HC) approximation in the frequency space has also been investigated with various basis functions: Fourier series \cite{G}, \cite{GH}, polynomial approximations generated from the Chebyshev-Gauss-Lobatto points \cite{BNR}, Jacobi polynomials \cite{SW}. 

Although the regular hyperbolic cross (RHC) approximation \eqref{RHC} reduces the effect of the ``curse of dimensionality" in some degree, the convergence rate is still deteriorated slowly with the dimension increasing (noting the term $(\log{N})^{(d-1)(r+1)}$ in the previous paragraph). To completely break the ``curse of dimensionality",  the optimized hyperbolic cross (OHC) approximation \eqref{OHC} is introduced in \cite{GH}. It has been shown in \cite{K} that the convergence rate of the OHC approximation with $\gamma\in (0,1)$ (see definition in \eqref{OHC index set}) with Fourier series is of $\mathcal{O}(N^{-r})$ in our notation, where the dimension enters the constant in front. The first purpose of this paper is to establish the error estimate for the HC approximations with the generalized Hermite functions in the weighted Korobov spaces $\K_{\ba,\bb}^m(\mathbb{R}^d)$, see \eqref{koborov space}. In particular, we obtain the following results for the RHC/OHC approximation with the generalized Hermite functions. 
\begin{theorem}\label{thm-1}
For any $u\in\K_{\ba,\bb}^m(\Rd)$, $0\leq l<m$, (and $0<\g\leq\frac lm$,)
\begin{align*}
	\inf_{U_N\in X_N (or\ X_{N,\gamma})}\left|\left|u-U_N\right|\right|_{\K^l_{\ba,\bb}(\Rd)(or\ \W_{\ba,\bb}^l(\Rd))}\leq 
C N^{\frac{l-m}2}|u|_{\K^m_{\ba,\bb}(\Rd)},\quad\forall\, 0\leq l< m,
\end{align*}
where $C$ is some constant depending on $\ba$, $l$, $m$ and $d$ (or $\gamma$), $X_N$ (or $X_{N,\gamma}$) is defined in \eqref{RHC} (or \eqref{OHC}), $\W_{\ba,\bb}^l(\Rd)$ and $\K^l_{\ba,\bb}(\Rd)$ are the Sobolev-type spaces \eqref{Sobolev-type space} and the weighted Korobov spaces \eqref{koborov space}, respectively.
\end{theorem}

We follow the error analysis developed in \cite{SW} to show Theorem \ref{thm-1}. But it is necessary to point out that there is a gap in the proof of Theorem 2.3, \cite{SW}. We circumvent this by more delicate analysis. 

We are also interested in the dimensional adaptive HC approximation. 
The following error estimate is obtained with respect to the dependence of dimensions.
\begin{theorem}[{\scshape Theorem} \ref{B.1}]
	For any $u\in\K_{\ba,\bb}^m(\Rd)$, for $0<l\leq m$, we have
	\begin{align*}
		\inf_{U_N\in X_N}\left|u-U_N\right|_{W^l_{\ba,\bb}(\Rd)}\lesssim|\ba|_\infty^{l-m}\left(N_1^{l-m}+N_2^{\frac{1-\g}{d-d_1-\g}(l-m)}\right)^{\frac12}|u|_{\K_{\ba,\bb}^m(\Rd)},
	\end{align*}
	where $X_N$ is defined in \eqref{dim adaptive HC}, $\g$ is in the definition of OHC \eqref{OHC index set}, and $N_1$, $d_1$, $N_2$ are clarified in \eqref{dim adaptive indices}.
\end{theorem}
 
To avoid the distraction of our main results, we leave the detailed proof of this theorem in Appendix B.

The second purpose of this paper is to study the application of the Galerkin-type HSM with the HC approximation to high-dimensional linear parabolic PDEs. The error estimates in appropriate weighted Korobov spaces are investigated under various conditions (cf. conditions $\bm{(C_1)}$-$\bm{(C_6)}$ in section 3). There also exist rich literatures of  the applications of sparse grids algorithm to solve equations. It has already been successfully applied to problems from the integral equations \cite{GOS}, to interpolation and approximation \cite{KW}, to the stochastic differential equations \cite{ST}, \cite{NTW}, to high dimensional integration problems from physics and finance \cite{GG}, and to the solutions to elliptic PDEs, \cite{Z}, \cite{SY}. As to the parabolic PDEs, they are treated with a wavelet-based sparse grid discretization in \cite{PS}. Besides the finite element approaches, they are also handled with finite differences on sparse grids \cite{Gr1998} and finite volumn schemes \cite{H}. Griebel and Oeltz \cite{GOe} proposed a space-time sparse grid technique, where the tensor product of one-dimensional multilevel basis in time and a proper multilevel basis in space have been employed. To our best knowledge, it is the first time in this paper that the Galerkin HSM with sparse grids algorithm is applied to parabolic PDEs, and the error estimates are obtained in the appropriate spaces.
\begin{theorem}
	Assume that conditions $\bm{(C_1)}$-$\bm{(C_3)}$ are satisfied, and the solution to the equation \eqref{PDE} $u\in L^\infty(0,T;\K_{\ba,\bb}^m(\Rd))\cap L^2(0,T;\K_{\ba,\bb}^m(\Rd))$, for $m>1$. Let $u_N$ be the approximate solution obtained by HSM \eqref{HSM}, then
	\begin{align*}
		||u-u_N||(t)\lesssim c^*N^{\frac{1-m}2},
	\end{align*}
where $c^*$ depends on $\ba$, the norms of $L^2(0,T;\K_{\ba,\bb}^m(\Rd))$ and  $L^\infty(0,T;\K_{\ba,\bb}^m(\Rd))$.
\end{theorem}
\begin{theorem}
	Assume that conditions $\bm{(C_3)}$-$\bm{(C_6)}$ are satisfied and the solution to the equation \eqref{PDE} $u\in L^2(0,T;\K_{\ba,\bb}^m(\Rd))$, for some integer $m>\max\{|\bg|_1,|\bd|_1+1\}$ ($\bg$, $\bd$ are two parameters in condition $\bm{(C_6)}$), and $u_N$ is the approximate solution obtained by HSM \eqref{HSM}, then 
\begin{align*}
	||u-u_N||(t)\lesssim c^{\sharp}N^{\frac{\max\{|\bg|_1,|\bd|_1+1\}-m}2},
\end{align*}
where $c^{\sharp}$ depends on $\ba$, $T$ and the norm of $L^2(0,T;\K_{\ba,\bb}^m(\Rd))$.
\end{theorem}

The paper is organized as following. The error analysis of the HC approximations with generalized Hermite functions is in section 2. Section 3 is devoted to the error estimate of HSM with HC approximation applying to linear parabolic PDE in suitable spaces under certain conditions. Finally, in section 4, the numerical experiment has been included to verify the exponential convergence of the HSM with the HC approximation to PDE. In the appendices, the error analysis of the full grid approximation and the dimensional adaptive HC approximation with generalized Hermite function are illustrated in detail.

\section{Hyperbolic cross approximation with generalized Hermite functions}

\setcounter{equation}{0}
\setcounter{theorem}{0}

\subsection{Notations}

Let us first clarify the notations to be used throughout the paper.
\begin{enumerate}
	\item[$\diamond$] Let \R (resp., \N) denote all the real numbers (resp., natural numbers), and let $\N_0=\N\cup\{0\}$.
	\item[$\diamond$] For any $d\in\N$, we use boldface lowercase letters to denote d-dimensional multi-indices and vectors, e.g., $\bk=(k_1,k_2,\ldots,k_d)\in\Nd$ and $\ba=(\a_1,\a_2,\ldots,\a_d)\in\Rd$.
	\item[$\diamond$] Let $\bm{1}=(1,1,\ldots,1)\in\mathbb{N}^d$, and let $\bm{e}_i=(0,\ldots,1,\ldots,0)$ be the $i^{\textup{th}}$ unit vector in \Rd. For any scalar $s\in\R$, we define the componentwise operations:
	\begin{align*}
		\ba\pm\bk=&(\a_1\pm k_1,\ldots,\a_d\pm k_d), \quad 
		\ba\pm s:=\ba\pm s\bm{1}=(\a_1\pm s,\ldots,\a_d\pm s),\\
		\frac1{\ba}=&\left(\frac1{\a_1},\ldots,\frac1{\a_d}\right), \quad
		\ba^{\bk} =\a_1^{k_1}\cdots\a_d^{k_d},
	\end{align*} 
and
	\begin{align*}
		\ba\geq\bk \Leftrightarrow \a_j\geq k_j,\quad\forall\, 1\leq j\leq d;\quad
		\ba\geq s \Leftrightarrow \a_j\geq s,\quad\forall\, 1\leq j\leq d. 
	\end{align*} 
	\item[$\diamond$] The frequently used norms are denoted as
	\begin{align*}
		|\bk|_1=\sum_{j=1}^dk_j;\quad
		|\bk|_\infty=\max_{1\leq j\leq d}k_j;\quad
		|\bk|_{\textup{mix}}=\prod_{j=1}^d\bar{k}_j,
	\end{align*}
	where $\bar{k}_j=\max\{1,k_j\}$.
	\item[$\diamond$] Given a multivariate function $u(\bx)$, we denote, the $\bk^{\textup{th}}$ mixed partial derivative by 
	\begin{align*}
		\bp_{\bx}^{\bk} u=\frac{\p^{|\bk|_1}u}{\p x_1^{k_1}\cdots\p x_d^{k_d}}
			=\p_{x_1}^{k_1}\cdots\p_{x_d}^{k_d}u.
	\end{align*}
	In particular, we denote $\bp_{\bx}^su=\bp_{\bx}^{s\bm{1}}u=\bp_{\bx}^{(s,s,\ldots,s)}u$.
	\item[$\diamond$] Let $L^2(\Rd)$ be the Lebesgue space in $\Rd$, equipped with the norm $||\cdot||=\left(\int_{\Rd}|\cdot|^2 d\bx\right)^{\frac12}$ and the
scalar product $\langle\cdot,\cdot\rangle$.
	\item[$\diamond$] We follow the convention in the asymptotic analysis,
$a\sim b$ means that there exist some constants $C_1,C_2>0$ such
that $C_1a\leq b\leq C_2a$; $a\lesssim b$ means that there exists
some constant $C_3>0$ such that $a\leq C_3b$; $N\gg1$ means that $N$ is sufficiently large.
	\item[$\diamond$] We denote $C$ as some generic positive constant,
which may vary from line to line.
\end{enumerate}

\subsection{Generalized Hermite functions and its properties}

Recall that the univariate physical Hermite polynomials $H_n(x)$ are given by
$H_n(x)=(-1)^ne^{x^2}\partial_x^ne^{-x^2}$, $n\geq0$. Two well-known and useful
facts of Hermite polynomials are the mutually orthogonality with respect to the weight $w(x)=e^{-x^2}$ and the three-term recurrence, i.e.,
\begin{align}\label{recurrence}
    H_0\equiv1;\quad H_1(x)=2x;\quad\textup{and}\quad
    H_{n+1}(x)=2xH_n(x)-2nH_{n-1}(x).
\end{align}
It is studied in \cite{T} that the scaling and translating factors are crucial to the resolution of Hermite functions. And the necessity of the translating factor is discussed in \cite{LY}. Let us define the generalized Hermite functions as
\begin{align}\label{new hermite}
    \H_n^{\a,\b}(x)=\left(\frac\a{2^nn!\sqrt{\pi}}\right)^{\frac12}H_n(\a(x-\b))e^{-\frac12\a^2(x-\b)^2},
\end{align}
for $n\geq0$, where $\alpha>0$ is the scaling factor, and $\beta\in\mathbb{R}$ is the translating factor. It is readily to derive the following properties for (\ref{new
hermite}):
\begin{enumerate}
    \item[$\diamond$]\label{orthogonal of H_n^alpha,beta} The $\{\H_n^{\a,\b}\}_{n\in\N_0}$ forms an orthonormal basis of
    $L^2(\mathbb{R})$, i.e.
        \begin{align}\label{orthogonal}
            \int_{\mathbb{R}}\H_n^{\a,\b}(x)\H_m^{\a,\b}(x)dx=\delta_{nm},
        \end{align}
    where $\delta_{nm}$ is the Kronecker function.
    \item[$\diamond$] \label{eigenvalue} $\H_n^{\a,\b}(x)$ is the $n^{\textup{th}}$
    eigenfunction of the following Strum-Liouville problem
        \begin{align}\label{S-L}
            e^{\frac12\a^2(x-\b)^2}\p_x(e^{-\a^2(x-\b)^2}\p_x(e^{\frac12\a^2(x-\b)^2}u(x)))+\lambda_nu(x)=0,
        \end{align}
    with the corresponding eigenvalue $\lambda_n=2\a^2n$.
    \item[$\diamond$] By convention, $\H_n^{\a,\b}\equiv0$, for $n<0$. For
    $n\geq0$, the three-term recurrence is inherited from the
    Hermite polynomials:
    \begin{align}\label{recurrence for RT hermite function}
        2\a^2(x-\b)\H_n^{\a,\b}(x)=\sqrt{\lambda_n}\H_{n-1}^{\a,\b}(x)+\sqrt{\lambda_{n+1}}\H_{n+1}^{\a,\b}(x).
    \end{align}
    \item[$\diamond$] The derivative of $\H_n^{\a,\b}(x)$ is explicitly
    expressed, namely
    \begin{align}\label{derivative}
        \p_x\H_n^{\a,\b}(x) =
        \frac12\sqrt{\lambda_n}\H_{n-1}^{\a,\b}(x)-\frac12\sqrt{\lambda_{n+1}}\H_{n+1}^{\a,\b}(x).
    \end{align}
    \item[$\diamond$] Let $\D_x=\p_x+\a^2(x-\b)$. Then 
	\begin{align}\label{Derivative}
		\D_x^k\H_n^{\a,\b}(x)=\sqrt{\mu_{n,k}}\H_{n-k}^{\a,\b}(x),\quad\forall\, n\geq k\geq1,
	\end{align}
	where 
	\begin{align}\label{mu part}
		\mu_{n,k}=\prod_{j=0}^{k-1}\lambda_{n-j}=\frac{2^k\a^{2k}n!}{(n-k)!},\quad\forall\,n\geq k\geq1.
	\end{align}
\item[$\diamond$] The orthogonality of $\{\D_x^k\H_n^{\a,\b}(x)\}_{n\in\N_0}$ holds, i.e.,
    \begin{align}\label{orthogonality of Derivative}
    	\int_{\mathbb{R}}\D_x^k\H_n^{\a,\b}(x)\D_x^k\H_m^{\a,\b}(x)dx =\mu_{n,k}\delta_{nm}.
    \end{align}
\end{enumerate}

For notational convenience, we extend $\mu_{n,k}$ in \eqref{mu part} for all $n,k\in\N_0$:
\begin{align}\label{mu}
	\mu_{n,k}=
	\left\{\begin{aligned}
		1,\quad&\textup{if}\ n\geq k, k=0,\\
		0,\quad&\textup{if}\ k>n\geq0.
	\end{aligned}\right.
\end{align}

Now we define the d-dimensional tensorial generalized Hermite functions as
\begin{align*}
	\bH_{\bn}^{\ba,\bb}(\bx)=\prod_{j=1}^d\H_{n_j}^{\a_j,\b_j}(x_j),
\end{align*}
for $\ba>0$, $\bb\in\Rd$ and $\bx\in\Rd$. It verifies readily that the properties \eqref{Derivative}-\eqref{orthogonality of Derivative} can be extended correspondingly to multivariate generalized Hermite functions. Let $\bD_{\bx}^{\bk}=\D_{x_1}^{k_1}\cdots\D_{x_d}^{k_d}$, then
\begin{align}\label{Derivative to d-dim}
	\bD_{\bx}^{\bk}\bH_{\bn}^{\ba,\bb}
		=\sqrt{\bm{\mu}_{\bn,\bk}}\bH_{\bn-\bk}^{\ba,\bb};
\end{align}
and
\begin{align}\label{orthogonal of Derivative to d-dim}
	\int_{\Rd}\bD_{\bx}^{\bk}\bH_{\bn}^{\ba,\bb}(\bx)\bD_{\bx}^{\bk}\bH_{\bm{m}}^{\ba,\bb}(\bx)d\bx=\bm{\mu}_{\bn,\bk}\bm{\delta}_{\bn\bm{m}},
\end{align}
for $\ba>0$, $\bb\in\Rd$, where 
\begin{align}\label{bm_mu}
	\bm{\mu}_{\bn,\bk}=\prod_{j=1}^d\mu_{n_j,k_j}\quad\textup{and}\quad
	\bm{\delta}_{\bn\bm{m}}=\prod_{j=1}^d\delta_{n_jm_j}.
\end{align}
Here, $\mu_{\cdot,\cdot}$ is defined in \eqref{mu part} and \eqref{mu}, and $\bd_{\bn\bm{m}}$ is the tensorial Kronecker function. 

The generalized Hermite
functions $\{\bH_{\bn}^{\ba,\bb}(\bx)\}_{\bn\in\Nd}$ form an orthonormal basis of $L^2(\Rd)$. That is, for any function $u\in L^2(\Rd)$, it can be written in the form
\begin{align}\label{Hermite representation}
    u(\bx)=\sum_{\bn\geq0}\hat{u}_{\bn}^{\ba,\bb}\H_{\bn}^{\ba,\bb}(\bx),\quad \textup{with}\quad \hat{u}_{\bn}^{\ba,\bb}=\int_{\Rd}u(\bx)\H_{\bn}^{\ba,\bb}(\bx)d\bx.
\end{align}

Hence, we have $\bD_{\bx}^{\bk}u(\bx)=\sum_{\bn\geq\bk}\hat{u}_{\bn}^{\ba,\bb}
\bD_{\bx}^{\bk}\H_{\bn}^{\ba,\bb}(\bx)$. Furthermore,
\begin{align}\label{norm in frequency}
	||\bD_{\bx}^{\bk}u||^2= \sum_{\bn\geq\bk}\bm{\mu}_{\bn,\bk}|\hat{u}_{\bn}^{\ba,\bb}|^2
		\overset{\eqref{mu}}=\sum_{\bn\in\Nd}\bm{\mu}_{\bn,\bk}|\hat{u}_{\bn}^{\ba,\bb}|^2.
\end{align}

\subsection{Multivariate orthogonal projection and approximations}

In this section, we aim to arrive at some typical error esitmate of the form
\begin{align*}
	\inf_{U_N\in X_N}||u-U_N||_l\lesssim N^{-c(l,r)}||u||_r,
\end{align*}
where $c(l,r)$ is some positive constant depending on $l$ and $r$, $||\cdot||_{l}$ is the norm of some functional space, $l$ indicates the regularity of the function in some sense, and $X_N$ is an approximation space. In this paper, $X_N$ is defined as 
\begin{align}\label{space X}
	X_N^{\ba,\bb}=\textup{span}\{\bH_{\bn}^{\ba,\bb}:\ \bn\in\bO_N\},
\end{align}
where $\bO_N\subset \Nd$ is some index set. With different choices of $\bO_N$, it yields full grid, RHC, OHC, etc.. 

Let us denote the orthogonal projection operator $P_N^{\ba,\bb}: L^2(\Rd)\rightarrow X_N^{\ba,\bb}$, i.e., for any $u\in L^2(\Rd)$,
\begin{align*}
	\langle(u-P_N^{\ba,\bb}u),v\rangle=0,\quad\forall\, v\in X_N^{\ba,\bb}.
\end{align*}
Or equivalently,
\begin{align}\label{PN in frequency}
	P_N^{\ba,\bb}u(\bx)=\sum_{\bn\in\bO_N}\hat{u}_{\bn}^{\ba,\bb}\bH_{\bn}^{\ba,\bb}(\bx).
\end{align}

We shall estimate how close the projected function $P_N^{\ba,\bb}u$ is to $u$, with respect to various index sets $\bO_N$ and norms.
\subsubsection{Appoximations on the full grid}
The index set $\bO_N$ corresponding to the d-dimensional full tensor grid is 
\begin{align*}
	\bO_N=\{\bn\in\Nd:\ |\bn|_\infty\leq N\}.
\end{align*}
And $X_N^{\ba,\bb}$ is defined in \eqref{space X}. Let us define the Sobolev-type space as
\begin{align}\label{Sobolev-type space}
	\W_{\ba,\bb}^m(\Rd)=\{u:\ \bD_{\bx}^{\bk}u\in L^2(\Rd), 0\leq|\bk|_1\leq m\},\quad \forall\,m\in\N_0,
\end{align}
equipped with the norm and seminorm
\begin{align}\label{sobolev norm}
	||u||_{\W_{\ba,\bb}^m(\Rd)}=&\left(\sum_{0\leq|\bk|_1\leq m}\left|\left|\bD_{\bx}^{\bk}u\right|\right|^2\right)^{\frac12},\\\label{sobolev seminorm}
	|u|_{\W_{\ba,\bb}^m(\Rd)}=&\left(\sum_{j=1}^d\left|\left|\D_{x_j}^mu\right|\right|^2\right)^{\frac12}.
\end{align}
It is clear that $\W_{\ba,\bb}^0(\Rd)=L^2(\Rd)$, and 
\begin{align}\label{seminorm in frequency}
	|u|_{\W_{\ba,\bb}^m(\Rd)}^2\overset{\eqref{norm in frequency}}=\sum_{j=1}^d\sum_{\bn\in\Nd}\mu_{n_j,m}\left|\hat{u}_{\bn}^{\ba,\bb}\right|^2.
\end{align}
\begin{theorem}\label{projection on full grid}
	Given $u\in\W_{\ba,\bb}^m(\Rd)$, we have for any $0\leq l\leq m$,
	\begin{align}
		\left|P_N^{\ba,\bb}u-u\right|_{\W_{\ba,\bb}^l(\Rd)}\lesssim|\ba|_\infty^{l-m}N^{\frac{l-m}2}|u|_{\W_{\ba,\bb}^m(\Rd)},
	\end{align}
	for $N\gg1$. Furthermore, 
\begin{align*}
	\left|\left|P_N^{\ba,\bb}u-u\right|\right|_{\W_{\ba,\bb}^l(\Rd)}\lesssim C_{\ba,l,m}N^{\frac{l-m}2}|u|_{\W_{\ba,\bb}^m(\Rd)},
\end{align*}
where $C_{\ba,l,m}$ is some constant depending on $\ba$, $l$ and $m$.
\end{theorem}
Since the proof of this theorem is similar to that in \cite{SW}, and to avoid the distraction of our main results, we put the proof in Appendix A. It is clear that the convergence rate deteriorates rapidly with respect to the cardinality of the full grid. That is,
\begin{align*}
	\left|\left|P_N^{\ba,\bb}u-u\right|\right|_{\W_{\ba,\bb}^l(\Rd)}\lesssim C_{\ba,l,m}M^{\frac{l-m}{2d}}|u|_{\W_{\ba,\bb}^m(\Rd)},
\end{align*}
where $M=\textup{card}(\bO_N)=(N+1)^d$.

\subsubsection{RHC approximation}
As we mentioned in the introduction, the HC approximation is an efficient tool to overcome the ``curse of dimensionality" in some degree. The index set of RHC approximation is $\bO_N=\{\bn\in\Nd:\ |\bn|_{\textup{mix}}\leq N\}$. It is known that the cardinality of $\bO_N$ is $\mathcal{O}(N(\ln{N})^{d-1})$ \cite{GH}. Correspondingly, the finite dimensional subspace $X_N^{\ba,\bb}$ is
\begin{align}\label{RHC}
	X_N^{\ba,\bb}=\textup{span}\{\bH_{\bn}^{\ba,\bb}:\ |\bn|_{\textup{mix}}\leq N\}.
\end{align}
Let the orthogonal projection operator $P_N^{\ba,\bb}: L^2(\Rd)\rightarrow X_N^{\ba,\bb}$ be defined before. Denote the $\bk-$complement of $\bO_N$ by 
\begin{align}\label{k-complement}
	\bO_{N,\bk}^c:=\{\bn\in\Nd:\ |\bn|_{\textup{mix}}>N\ \textup{and}\ \bn\geq\bk\},\quad\forall\,\bk\in\Nd.
\end{align}
We define the Koborov-type space as
\begin{align}\label{koborov space}
	\K_{\ba,\bb}^r(\Rd)=\{u:\ \D_{\bx}^{\bk}u\in L^2(\Rd),\ 0\leq|\bk|_\infty\leq r\},\quad\forall\, m\in\Nd,
\end{align}
equipped with the norm and seminorm
\begin{align}\label{Koborov norm}
	||u||_{\K_{\ba,\bb}^r(\Rd)}=&\left(\sum_{0\leq|\bk|_\infty\leq r}\left|\left|\D_{\bx}^{\bk}u\right|\right|^2\right)^{\frac12},\\\label{Koborov seminorm}
	|u|_{\K_{\ba,\bb}^r(\Rd)}=&\left(\sum_{|\bk|_\infty=r}\left|\left|\D_{\bx}^{\bk}u\right|\right|^2\right)^{\frac12}.
\end{align}
\begin{remark}
	It is easy to see from the definitions that $\K_{\ba,\bb}^0(\Rd)=L^2(\Rd)$ and $\W_{\ba,\bb}^{dl}(\Rd)\subset\K_{\ba,\bb}^{l}(\Rd)\subset\W_{\ba,\bb}^{l}(\Rd)$.
\end{remark}
\begin{theorem}\label{regular HC}
	Given $u\in\K_{\ba,\bb}^m(\Rd)$, for $0\leq \bl\leq m$, we have
	\begin{align*}
		\left|\left|\D_{\bx}^{\bl}\left(P_N^{\ba,\bb}u-u\right)\right|\right|
			\leq C_{\ba,\bl,m,d}N^{\frac{|\bl|_\infty-m}2}|u|_{\K_{\ba,\bb}^m(\Rd)},
	\end{align*}
where $C_{\ba,\bl,m,d}$ is some constant depending on $\ba$, $\bl$, $m$ and $d$, for $N\gg1$ (more precisely, at least $N>m^d$). In particular, if $\ba=\bm{1}$, then
\begin{align*}
	C_{\bm{1},\bl,m,d}=2^{|\bl|_\infty-m}m^{(2d-1)m-|\bl|_1-(d-1)|\bl|_\infty}.
\end{align*}
\end{theorem}
\begin{proof}
	From \eqref{PN in frequency}, \eqref{norm in frequency}, we have
\begin{align*}
	\left|\left|\D_{\bx}^{\bl}(P_N^{\ba,\bb}u-u)\right|\right|^2
		=&\sum_{\bn\in\bO_N^c}\bm{\mu}_{\bn,\bl}\left|\hat{u}_{\bn}^{\ba,\bb}\right|^2
		=\sum_{\bn\in\bO_{N,m}^c}\bm{\mu}_{\bn,\bl}\left|\hat{u}_{\bn}^{\ba,\bb}\right|^2
		+\sum_{\bn\in\bO_{N,\bl}^c\setminus\bO_{N,m}^c}\bm{\mu}_{\bn,\bl}\left|\hat{u}_{\bn}^{\ba,\bb}\right|^2\\
		:=&\rom{2}_1+\rom{2}_2.
\end{align*}
For $\rom{2}_1$:
\begin{align*}
	\rom{2}_1\leq\max_{\bn\in\bO_{N,m}^c}\left\{\frac{\bm{\mu}_{\bn,\bl}}{\bm{\mu}_{\bn,m}}\right\}\sum_{\bn\in\bO_{N,m}^c}\bm{\mu}_{\bn,m}\left|\hat{u}_{\bn}^{\ba,\bb}\right|^2.
\end{align*}
With the facts that 
\begin{align}\label{thm2.2_eq1}\notag
	\frac{\bm{\mu}_{\bn,\bl}}{\bm{\mu}_{\bn,m}}
		=&2^{|\bl|_1-dm}\prod_{j=1}^d\a_j^{2(l_j-m)}\prod_{j=1}^d\frac1{(n_j-l_j)\cdots(n_j-m+1)}\\\notag
		=&2^{|\bl|_1-dm}\prod_{j=1}^d\a_j^{2(l_j-m)}\prod_{j=1}^dn_j^{l_j-m}\prod_{j=1}^d\left(1-\frac{l_j}{n_j}\right)^{-1}\cdots\left(1-\frac{m-1}{n_j}\right)^{-1}\\
		\overset{\eqref{k-complement}}\leq&2^{|\bl|_1-dm}\prod_{j=1}^d\a_j^{2(l_j-m)}N^{|\bl|_\infty-m}\prod_{j=1}^d\left(1-\frac{l_j}{n_j}\right)^{-1}\cdots\left(1-\frac{m-1}{n_j}\right)^{-1}
\end{align}
and 
\begin{align}\label{thm2.2_eq3}\notag
	\max_{\bn\in\bO_{N,m}^c}\left\{\prod_{j=1}^d\left(1-\frac{l_j}{n_j}\right)^{-1}\cdots\left(1-\frac{m-1}{n_j}\right)^{-1}\right\}
	\leq&\max_{\bn\in\bO_{N,m}^c}\left\{\prod_{j=1}^d\left(1-\frac{m-1}{n_j}\right)^{l_j-m}\right\}\\
		\leq&\prod_{j=1}^dm^{m-l_j}=m^{dm-|\bl|_1},		
\end{align}
we arrive that 
\begin{align}\label{thm2.2_II1}
	\rom{2}_1\leq\left(\frac m2\right)^{dm-|\bl|_1}\prod_{j=1}^d\a_j^{2(l_j-m)}N^{|\bl|_\infty-m}\left|\left|\D_{\bx}^{m\cdot\bm{1}}u\right|\right|^2.
\end{align}
For $\rom{2}_2$: The index set $\bO_{N,\bl}^c\setminus\bO_{N,m}^c$ is 
\begin{align*}
	\bO_{N,\bl}^c\setminus\bO_{N,m}^c=\{\bn\in\Nd:\ |\bn|_{\textup{mix}}>N\ \textup{and}\ \bn\geq\bl,\ \exists\, j,\ \textup{such that}\  n_j<m\}.
\end{align*}
Let us divide the index $1\leq j\leq d$ into two parts
\begin{align}\label{cN and cNc}
	\cN:=\{j:\ l_j\leq n_j<m,\, 1\leq j\leq d\},\quad\cN^c:=\{j:\ n_j\geq m,\,1\leq j\leq d\}.
\end{align}
It is easy to see that neither $\cN$ nor $\cN^c$ is empty set. We denote
\begin{align}\label{thm2.2_eq2}
	\tilde{\bm{\mu}}_{\bn,\bl,m}
	=\left(\prod_{j\in\cN}\mu_{n_j,l_j}\right)\left(\prod_{i\in\cN^c}\mu_{n_i,m}\right)
		:=\bm{\mu}_{\bn,\bk},
\end{align}
where $\bk$ is a d-dimensional index consisting of $l_j$ for $j\in\cN$ and $m$ for $j\in\cN^c$. Now, we treat $\rom{2}_2$ as 
\begin{align}\label{thm2.2_II2}
	\rom{2}_2\overset{\eqref{thm2.2_eq2}}\leq&\max_{\bn\in\bO_{N,\bl}^c\setminus\bO_{N,m}^c}\left\{\frac{\bm{\mu}_{\bn,\bl}}{\bm{\mu}_{\bn,\bk}}\right\}\sum_{\bn\in\bO_{N,\bl}^c\setminus\bO_{N,m}^c}\bm{\mu}_{\bn,\bk}\left|\hat{u}_{\bn}^{\ba,\bb}\right|^2
		\leq\max_{\bn\in\bO_{N,\bl}^c\setminus\bO_{N,m}^c}\left\{\frac{\bm{\mu}_{\bn,\bl}}{\bm{\mu}_{\bn,\bk}}\right\}|u|_{\K_{\ba,\bb}^m(\Rd)}^2,
\end{align}
since $|\bk|_\infty=m$. It remains to estimate the maximum in \eqref{thm2.2_II2}.
\begin{align}\label{thm2.2_eq5}\notag
	\frac{\bm{\mu}_{\bn,\bl}}{\bm{\mu}_{\bn,\bk}}
		=&2^{|\bl|_1-|\bk|_1}\prod_{j\in\cN^c}\a_j^{2(l_j-m)}\frac1{(n_j-l_j)\cdots(n_j-m+1)}\\
		=&2^{|\bl|_1-|\bk|_1}\prod_{j\in\cN^c}\a_j^{2(l_j-m)}\prod_{j\in\cN^c}n_j^{l_j-m}\prod_{j\in\cN^c}\left(1-\frac{l_j}{n_j}\right)^{-1}\cdots\left(1-\frac{m-1}{n_j}\right)^{-1}.
\end{align}
Observe that $j\in\cN^c$ implies $n_j\geq m>\bl\geq0$. That is, $n_j\geq1$. Hence, $\bar{n}_j=n_j$, for all $j=1,\cdots,d$. In view of $|\bn|_{\textup{mix}}>N$, we deduce that 
\begin{align*}
	\prod_{j\in\cN^c}\bar{n}_j>\frac N{\prod_{j\in\cN}\bar{n}_j}>\frac N{\prod_{j\in\cN}m}.
\end{align*}
With the same estimate in \eqref{thm2.2_eq3} and the fact that 
\begin{align}\label{thm2.2_eq6}
	2^{|\bl|_1-|\bk|_1}=2^{\sum_{j\in\cN^c}(l_j-m)}\leq2^{|\bl|_\infty-m},
\end{align}
it yields that
\begin{align}\label{thm2.2_eq4}
	\max_{\bn\in\bO_{N,\bl}^c\setminus\bO_{N,m}^c}\left\{\frac{\bm{\mu}_{\bn,\bl}}{\bm{\mu}_{\bn,\bk}}\right\}
		\leq C_{\ba,\bl,m}2^{|\bl|_\infty-m}m^{(2d-1)m-|\bl|_1-(d-1)|\bl|_\infty}N^{|\bl|_\infty-m},
\end{align}
where $C_{\ba,\bl,m}$ denotes some constant depending on $\ba$, $\bl$ and $m$. The desired result follows immediately from \eqref{thm2.2_II1}, \eqref{thm2.2_II2} and \eqref{thm2.2_eq4}.
\end{proof}

\begin{corollary}\label{coro2.1}
\begin{align*}
	\left|\left|P_N^{\ba,\bb}u-u\right|\right|_{\K^l_{\ba,\bb}(\Rd)}\leq 
C_{\ba,l,m,d}N^{\frac{l-m}2}|u|_{\K^m_{\ba,\bb}(\Rd)},\quad\forall\, 0\leq l\leq m,
\end{align*}
where $C_{\ba,l,m,d}$ is some constant depending on $\ba$, $l$, $m$ and $d$.
\end{corollary}
\begin{remark}
	Recall that $M=\textup{card}(\bO_N)=\mathcal{O}(N(\ln{N})^{d-1})\leq CN^{1+\epsilon(d-1)}$, for arbitrary small $\epsilon>0$. Then
\begin{align*}
	\left|\left|P_N^{\ba,\bb}u-u\right|\right|_{\K^l_{\ba,\bb}(\Rd)}\leq 
C_{\ba,l,m,d}M^{\frac{l-m}{2(1+\epsilon(d-1))}}|u|_{\K^m_{\ba,\bb}(\Rd)},\quad\forall\, 0\leq l\leq m,
\end{align*}
where $C_{\ba,l,m,d}$ is some constant depending on $\ba$, $l$, $m$ and $d$. It is clear to see that the convergence rate deteriorates slightly with increasing $d$. 
\end{remark}

\subsubsection{OHC approximation}

In order to completely break the curse of dimensionality, we consider the index set introduced in \cite{GH}
\begin{align}\label{OHC index set}
	\bO_{N,\g}:=\{\bn\in\Nd:\ |\bn|_{\textup{mix}}|\bn|_\infty^{-\g}\leq N^{1-\g}\},\quad -\infty\leq\g<1.
\end{align}
The cardinality of $\bO_{N,\g}$ is $\mathcal{O}(N)$, for $\g\in(0,1)$, where the dependence of dimension is in the big-O, see \cite{GH}. The family of spaces are defined as
\begin{align}\label{OHC}
	X_{N,\g}^{\ba,\bb}:=\textup{span}\{\bH_{\bn}^{\ba,\bb}:\ \bn\in\bO_{N,\g}\}.
\end{align}
\begin{remark}
	In particular, we have $X_{N,0}^{\ba,\bb}=X_N^{\ba,\bb}$ in RHC \eqref{RHC}, and $X_{N,-\infty}^{\ba,\bb}=\textup{span}\{\H_{\bn}^{\ba,\bb}:\,|\bn|_{\infty}\leq N\}$, i.e. the full grid. 
\end{remark}
We denote the projection operator as $P_{N,\g}^{\ba,\bb}:\ L^2(\Rd)\rightarrow X_{N,\g}^{\ba,\bb}$. In this case, the $\bk-$complement of index set of $\bO_{N,\g}$ is 
\begin{align}\label{k complement of mix}
	\bO_{N,\g,\bk}^c=\{\bn\in\Nd:\ \bn\in\bO_{N,\g}^c\ \textup{and}\ \bn\geq\bk\},\quad\forall\,\bk\in\Nd.
\end{align}

Although \cite{SW} obtains the similar result for Jacobi polynomials as Theorem \ref{thm-OHC} below, we believe that there is a gap in their error analysis of OHC, namely Theorem 2.3, \cite{SW}. We circumvent it with more delicate analysis.
\begin{theorem}\label{thm-OHC}
	For any $u\in\K_{\ba,\bb}^m(\Rd)$, $d\geq2$, and $0\leq|\bl|_1<m$, 
	\begin{align}\label{estimate of thm2.3}
		\left|\left|\bD_{\bx}^{\bl}\left(P_{N,\g}^{\ba,\bb}u-u\right)\right|\right|
	\leq C_{\ba,\bl,m,d,\g}|u|_{\K_{\ba,\bb}^m(\Rd)}
		\left\{\begin{aligned}
			N^{\frac{|\bl|_1-m}2},\quad&\textup{if}\ 0<\g\leq\frac {|\bl|_1}m\\
			N^{\frac{(1-\g)[|\bl|_1-(d-1)m]}{d-1-\g}},\quad&\textup{if}\ \frac{|\bl|_1}m\leq \g<1,
		\end{aligned}\right.
	\end{align}
where $C_{\ba,\bl,m,d,\g}$ is some constant depending on $\ba$, $\bl$, $m$, $d$ and $\g$. In particular, if $\ba=\bm{1}$, then
\begin{equation*}
	C_{\bm{1},\bl,m,d,\g}=m^{dm-|\bl|_1}
	\left\{\begin{aligned}
		2^{|\bl|_\infty-m}m^{\frac{(d-1)(\g m-|\bl|_1)}{1-\g}},&\quad\textup{if}\ 0<\g\leq\frac{|\bl|_1}m\\
		2^{|\bl|_1-dm},&\quad\textup{if}\ \frac{|\bl|_1}m\leq\g<1.
	\end{aligned}\right.
\end{equation*}
\end{theorem}
\begin{proof}
	As argued in the proof of Theorem \ref{regular HC}, we arrive
\begin{align}\label{thm2.3_eq0}\notag
	\left|\left|\bD_{\bx}^{\bl}\left(P_{N,\g}^{\ba,\bb}u-u\right)\right|\right|^2
		\leq&\max_{\bn\in\bO_{N,\g,m}^c}\left\{\frac{\bm{\mu}_{\bn,\bl}}{\bm{\mu}_{\bn,m}}\right\}\sum_{\bn\in\bO_{N,\g,m}^c}\bm{\mu}_{\bn,m}\left|\hat{u}_{\bn}^{\ba,\bb}\right|^2\\\notag
		&+\max_{\bn\in\bO_{N,\g,\bl}^c\setminus\bO_{N,\g,m}^c}\left\{\frac{\bm{\mu}_{\bn,\bl}}{\tilde{\bm{\mu}}_{\bn,\bl,m}}\right\}\sum_{\bn\in\bO_{N,\g,\bl}^c\setminus\bO_{N,\g,m}^c}\tilde{\bm{\mu}}_{\bn,\bl,m}\left|\hat{u}_{\bn}^{\ba,\bb}\right|^2\\
		:=&\rom{3}_1+\rom{3}_2,
\end{align}
where $\tilde{\bm{\mu}}_{\bn,\bl,m}$ is defined as in \eqref{thm2.2_eq2}. To estimate $\rom{3}_1$, like in \eqref{thm2.2_eq1}, we have
\begin{align}\label{thm2.3_eq1}\notag
	\frac{\bm{\mu}_{\bn,\bl}}{\bm{\mu}_{\bn,m}}
		=&2^{|\bl|_1-dm}\prod_{j=1}^d\a_j^{2(l_j-m)}\prod_{j=1}^d\left(1-\frac{l_j}{n_j}\right)^{-1}\cdots\left(1-\frac{m-1}{n_j}\right)^{-1}\prod_{j=1}^dn_j^{l_j-m}\\
		:=& D_1\prod_{j=1}^dn_j^{l_j-m}.
\end{align}
The estimate of $\max_{\bn\in\bO_{N,\g,m}^c}D_1$ is followed by the similar argument in \eqref{thm2.2_eq3}, i.e.,
\begin{align}\label{thm2.3_D1}
	\max_{\bn\in\bO_{N,\g,m}^c}D_1\leq \left(\frac m2\right)^{dm-|\bl|_1}\prod_{j=1}^d\a_j^{2(l_j-m)}.
\end{align}
Notice that for any $\bn\in\bO_{N,\g}^c$, 
\begin{align}\label{thm2.3_eq2}
	|\bn|_{\textup{mix}}|\bn|_\infty^{-\g}>N^{1-\g}\Rightarrow\left(\frac{|\bn|_\infty^{\g}}{|\bn|_{\textup{mix}}}\right)^{\frac1{1-\g}}<\frac1N
\end{align} 
and furthermore, if $\bn\in\bO_{N,\g,m}^c$,
\begin{align}\label{thm2.3_eq3}
	\frac{|\bn|_\infty}{|\bn|_{\textup{mix}}}\leq\frac1{m^{d-1}}.
\end{align}
Moreover, 
\begin{align}\label{thm2.3_eq7}
	|\bn|_\infty^{d-\g}\geq|\bn|_{\textup{mix}}|\bn|_\infty^{-\g}>N^{1-\g}
		\Rightarrow|\bn|_\infty>N^{\frac{1-\g}{d-\g}}.
\end{align}
Let us estimate the product on the right-hand side of \eqref{thm2.3_eq1}:
\begin{align}\label{thm2.3_eq4}
	\prod_{j=1}^dn_j^{l_j-m}
		=\left(\prod_{j=1}^dn_j^{l_j}\right)\left(\prod_{j=1}^dn_j\right)^{-m}
		\leq\left(\prod_{j=1}^d|\bn|_\infty^{l_j}\right)|\bn|_{\textup{mix}}^{-m}
		=|\bn|_\infty^{|\bl|_1}|\bn|_{\textup{mix}}^{-m}.
\end{align}
If $0<\g\leq\frac{|\bl|_1}m$, then
\begin{align}\label{thm2.3_eq5}\notag
	\max_{\bn\in\bO_{N,\g,m}^c}\prod_{j=1}^dn_j^{l_j-m}
		&\overset{\eqref{thm2.3_eq4}}\leq\max_{\bn\in\bO_{N,\g,m}^c}\left\{\left(\frac{|\bn|_\infty^{\g}}{|\bn|_{\textup{mix}}}\right)^{\frac{m-|\bl|_1}{1-\g}}\left(\frac{|\bn|_\infty}{|\bn|_{\textup{mix}}}\right)^{\frac{|\bl|_1-\g m}{1-\g}}\right\}\\
		&\overset{\eqref{thm2.3_eq2},\eqref{thm2.3_eq3}}<m^{\frac{(d-1)(\g m-|\bl|_1)}{1-\g}}N^{|\bl|_1-m}.
\end{align}
Otherwise, if $\frac{|\bl|_1}m\leq\g<1$, then
\begin{align}\label{thm2.3_eq6}
	\max_{\bn\in\bO_{N,\g,m}^c}\prod_{j=1}^dn_j^{l_j-m}
		\overset{\eqref{thm2.3_eq4}}\leq& \max_{\bn\in\bO_{N,\g,m}^c}\left\{\left(\frac{|\bn|_\infty^{\g}}{|\bn|_{\textup{mix}}}\right)^m|\bn|_\infty^{|\bl|_1-\g m}\right\}
		\overset{\eqref{thm2.3_eq2},\eqref{thm2.3_eq7}}\leq N^{\frac{1-\g}{d-\g}(|\bl|_1-\g m)-(1-\g)m}.
\end{align}
Combine \eqref{thm2.3_D1}, \eqref{thm2.3_eq5} and \eqref{thm2.3_eq6}, the first term on the right-hand side of \eqref{thm2.3_eq0} has the upper bound
\begin{align}\label{thm2.3_III1}
		\rom{3}_1\leq\left(\frac m2\right)^{dm-|\bl|_1}\prod_{j=1}^d\a_j^{2(l_j-m)}\left|\left|\bD_{\bx}^{m\cdot\bm{1}}u\right|\right|^2
		\left\{\begin{aligned}
			m^{\frac{(d-1)(\g m-|\bl|_1)}{1-\g}}N^{|\bl|_1-m},\quad&\textup{if}\ 0<\g\leq\frac {|\bl|_1}m\\
			N^{\frac{1-\g}{d-\g}(|\bl|_1-\g m)-(1-\g)m},\quad&\textup{if}\ \frac{|\bl|_1}m\leq \g<1.
		\end{aligned}\right.
\end{align}
Next, we consider $\rom{3}_2$. Define $\cN$ and $\cN^c$ as in \eqref{cN and cNc}. Like in \eqref{thm2.2_II2}, we obtain that
\begin{align}\label{thm2.3_eq15}
	\rom{3}_2\leq \max_{\bn\in\bO_{N,\g,\bl}^c\setminus\bO_{N,\g,m}^c}\left\{\frac{\bm{\mu}_{\bn,\bl}}{\bm{\mu}_{\bn,\bk}}\right\}|u|_{\K_{\ba,\bb}^m(\Rd)}^2.
\end{align}
We need to estimate the maximum similarly as in \eqref{thm2.2_eq5}:
\begin{align}\label{thm2.3_eq12}
	\frac{\bm{\mu}_{\bn,\bl}}{\bm{\mu}_{\bn,\bk}}
		=2^{|\bl|_1-|\bk|_1}\prod_{j\in\cN^c}\a_j^{2(l_j-m)}\prod_{j\in\cN^c}n_j^{l_j-m}\prod_{j\in\cN^c}\left(1-\frac{l_j}{n_j}\right)^{-1}\cdots\left(1-\frac{m-1}{n_j}\right)^{-1}
		:=D_2\prod_{j\in\cN^c}n_j^{l_j-m}.
\end{align}
Similar argument as in \eqref{thm2.2_eq3} yields that
\begin{align}\label{thm2.3_D2}
	\max_{\bn\in\bO_{N,\g,\bl}^c\setminus\bO_{N,\g,m}^c}D_2	
		\leq 2^{|\bl|_1-|\bk|_1}\prod_{j\in\cN^c}\a_j^{2(l_j-m)}m^{dm-\left|\tilde{\bl}\right|_1},
\end{align}
where
\begin{align}\label{thm2.3_tildevector}
	\tilde{\bl}=(l_1,\cdots,l_d)=
	\left\{\begin{aligned}
		l_j,\quad&\textup{if}\ j\in\cN^c\\
		0,\quad&\textup{otherwise}.
	\end{aligned}\right.
\end{align}
And it is verified that
\begin{align}\label{thm2.3_eq9}
	\prod_{j\in\cN^c}n_j^{l_j-m}
		\leq\left(\prod_{j\in\cN^c}|\tilde{\bn}|_\infty^{l_j}\right)\left(\prod_{j\in\cN^c} n_j\right)^{-m}
		= |\tilde{\bn}|_\infty^{\left|\tilde{\bl}\right|_1}|\tilde{\bn}|^{-m}_{\textup{mix}}
		\leq|\tilde{\bn}|_\infty^{|\bl|_1}|\tilde{\bn}|_{\textup{mix}}^{-m},
\end{align}
where $\tilde{\bn}$ is defined similarly as $\tilde{\bl}$ in \eqref{thm2.3_tildevector}. With similar argument as in \eqref{thm2.3_eq2}, we deduce that for any $\bn\in\bO_{N,\g}^c$,
\begin{align}\label{thm2.3_eq8}
	N^{1-\g}<|\bn|_{\textup{mix}}|\bn|_\infty^{-\g}\leq m^{d-1}|\tilde{\bn}|_{\textup{mix}}|\tilde{\bn}|_\infty^{-\g}
		\Rightarrow \left(\frac{|\tilde{\bn}|_\infty^{\g}}{|\tilde{\bn}|_{\textup{mix}}}\right)^{\frac1{1-\g}}<m^{\frac{d-1}{1-\g}}N^{-1}.
\end{align}
And similarly as in \eqref{thm2.3_eq3}, we have for any $\bn\in\bO_{N,\g,m}^c$,
\begin{align}\label{thm2.3_eq10}
	\frac{|\tilde{\bn}|_\infty}{|\tilde{\bn}|_{\textup{mix}}}\leq \frac1{m^{d-2}},
\end{align}
and 
\begin{align}\label{thm2.3_eq11}
	N^{1-\g}\overset{\eqref{thm2.3_eq10}}<m^{d-1}|\tilde{\bn}|_{\textup{mix}}|\tilde{\bn}|_\infty^{-\g}
		\leq m^{d-1}|\tilde{\bn}|_\infty^{d-1-\g}
	\Rightarrow |\tilde{\bn}|_\infty>\left(\frac{N^{1-\g}}{m^{d-1}}\right)^{\frac1{d-1-\g}}.
\end{align}
If $0<\g\leq\frac{|\bl|_1}m$, then
\begin{align}\label{thm2.3_eq13}\notag
	\max_{\bn\in\bO_{N,\g,\bl}^c\setminus\bO_{N,\g,m}^c}\prod_{j\in\cN^c}n_j^{l_j-m}
		&\overset{\eqref{thm2.3_eq9}}<\max_{\bn\in\bO_{N,\g,\bl}^c\setminus\bO_{N,\g,m}^c}\left\{\left(\frac{|\tilde{\bn}|_\infty^{\g}}{|\tilde{\bn}|_{\textup{mix}}}\right)^{\frac{m-|\bl|_1}{1-\g}}\left(\frac{|\tilde{\bn}|_\infty}{|\tilde{\bn}|_{\textup{mix}}}\right)^{\frac{|\bl|_1-\g m}{1-\g}}\right\}\\
		&\overset{\eqref{thm2.3_eq8},\eqref{thm2.3_eq10}}\leq m^{\frac1{1-\g}\{[(\g+1)d-(2\g+1)]m-(2d-3)|\bl|_1\}}N^{|\bl|_1-m}.
\end{align}
Otherwise, if $\frac{|\bl|_1}m\leq \g<1$, then
\begin{align}\label{thm2.3_eq14}\notag
	\max_{\bn\in\bO_{N,\g,\bl}^c\setminus\bO_{N,\g,m}^c}\prod_{j\in\cN^c}n_j^{l_j-m}
		&\overset{\eqref{thm2.3_eq9}}<\max_{\bn\in\bO_{N,\g,\bl}^c\setminus\bO_{N,\g,m}^c}\left\{\left(\frac{|\tilde{\bn}|_\infty^{\g}}{|\tilde{\bn}|_{\textup{mix}}}\right)^m|\tilde{\bn}|_\infty^{|\bl|_1-\g m}\right\}\\
		&\overset{\eqref{thm2.3_eq8},\eqref{thm2.3_eq11}}\leq m^{(d-1)\left[m-\frac{|\bl|_1-\g m}{d-1-\g}\right]}N^{\frac{(1-\g)[|\bl|_1-(d-1)m]}{d-1-\g}}.
\end{align}
Combine \eqref{thm2.2_eq5}, \eqref{thm2.3_eq15}, \eqref{thm2.3_D2}, \eqref{thm2.3_eq13} and \eqref{thm2.3_eq14}, we arrive
\begin{align}\label{thm2.3_III2}
	\rom{3}_2\leq& 2^{|\bl|_\infty-m}\prod_{j\in\cN^c}\a_j^{2(l_j-m)}m^{dm-|\bl|_1}|u|_{\K_{\ba,\bb}^m(\Rd)}^2\\\notag
	&\left\{\begin{aligned}
		m^{\frac1{1-\g}\{[(\g+1)d-(2\g+1)]m-(2d-3)|\bl|_1\}}N^{|\bl|_1-m},\quad&\textup{if}\ 0<\g\leq\frac{|\bl|_1}m\\
		m^{(d-1)\left[m-\frac{|\bl|_1-\g m}{d-1-\g}\right]}N^{\frac{(1-\g)[|\bl|_1-(d-1)m]}{d-1-\g}},\quad&\textup{if}\ \frac{|\bl|_1}m\leq\g<1.
	\end{aligned}\right.
\end{align}
Therefore, the desired result follows immediately from \eqref{thm2.3_III1} and \eqref{thm2.3_III2}.
\end{proof}

\begin{corollary}\label{coro2.2}
	For any $u\in\K_{\ba,\bb}^m(\Rd)$, $0\leq l<m$, and $0<\g\leq\frac lm$,
	\begin{align*}
		\left|\left|P_{N,\g}^{\ba,\bb}u-u\right|\right|_{\W_{\ba,\bb}^l(\Rd)}
			\leq C_{\ba,\bl,m,d,\g}N^{\frac{l-m}2}|u|_{\K_{\ba,\bb}^m(\Rd)}.
	\end{align*}
where $C_{\ba,\bl,m,d,\g}$ is some constant depending on $\ba$, $\bl$, $m$, $d$ and $\g$.
\end{corollary}
\begin{remark}
	Due to the fact that $M=\textup{card}(\bO_{N,\g})=\mathcal{O}(N)\leq CN$, we obtain that
\begin{align*}
		\left|\left|P_{N,\g}^{\ba,\bb}u-u\right|\right|_{\W_{\ba,\bb}^l(\Rd)}
			\leq C_{\ba,\bl,m,d,\g}M^{\frac{l-m}2}|u|_{\K_{\ba,\bb}^m(\Rd)}.
	\end{align*}
where $C_{\ba,\bl,m,d,\g}$ is some constant depending on $\ba$, $\bl$, $m$, $d$ and $\g$. It is clear to see that the convergence rate does not deteriorate with respect to $d$ anymore. The effect of the dimension goes into the constant in front.
\end{remark}

\section{Application to linear parabolic PDE}

\setcounter{equation}{0}
\setcounter{theorem}{0}

In this section, we shall study the Galerkin HSM with the HC approximation applying to high dimensional linear parabolic PDE. Let us consider the linear parabolic PDE of the general form:
\begin{align}\label{PDE}
	\left\{\begin{aligned}
		\p_tu(\bx)+Lu(\bx)=&f(\bx,t),\quad\bx\in\Rd,\ t\in[0,T]\\
		u(\bx,0)=&u_0(\bx),
	\end{aligned}\right.
\end{align}
where 
\begin{align}\label{L}
	Lu=-\nabla\cdot(\bm{A}\nabla u)+\bm{b}\cdot\nabla u+cu,
\end{align}
with $\bm{A}=(a_{ij})_{i,j=1}^d:\ \Rd\mapsto\R^{d\times d}$, $\bm{b}=(b_i)_{i=1}^d:\ \Rd\mapsto\Rd$ and $c:\ \Rd\mapsto\R$. The aim of HSM is to find $u_N\in X$, such that
\begin{align}\label{HSM}
	\langle\p_tu_N,\varphi\rangle -\A(u_N,\varphi)=\langle f,\varphi\rangle,\quad\forall\,\varphi\in X,
\end{align}
where $X$ is some approximate space, and $\A(u,v)$ is a bilinear form given by
\begin{align}\label{bilinear form}
	\A(u,v)=\int_{\Rd}(\nabla u)^T\bm{A}\nabla v+v\bm{b}\cdot\nabla u+cuv\ d\bx.
\end{align}
In our content, $X$ could be chosen as $X_N^{\ba,\bb}$, $X_{N,\g}^{\ba,\bb}$ in the previous section.

To guarantee the existence and regularity of the solution to \eqref{PDE}, we assume that 
\begin{enumerate}
	\item[$\bm{(C_1)}$] The bilinear form is continuous, i.e., there is a constant $C>0$ such that
\begin{align}\label{continuous}
	|\A(u,v)|\leq C||u||_{H_0^1(\Rd)}||v||_{H_0^1(\Rd)},\quad\forall\, u,v\in H_0^1(\Rd).
\end{align}
	\item[$\bm{(C_2)}$] The bilinear form is coercive, i.e., there exists some $c>0$ such that 
\begin{align}\label{coercive}
	\A(u,u)\geq c||u||^2_{H_0^1(\Rd)},\quad\forall\, u\in H_0^1(\Rd).
\end{align}
	\item[$\bm{(C_3)}$] The coefficients $a_{ij}$, $b_i$ and $c$ are smooth.
\end{enumerate}
Here, $H_0^1(\Rd)$ denotes the normal Sobolev space with the functions decaying to zero at infinity. More generally, $H_0^m(\mathbb{R}^d)$ is defined as, for any $u\in H^m(\Rd)$, it satisfies $|u|\rightarrow0$, as $|\bx|\rightarrow\infty$ and 
\begin{align}\label{H-space}
	||u||_{H^m(\Rd)}^2=\sum_{0\leq|\bk|_1\leq m}\left|\left|\bp_{\bx}^{\bk}u\right|\right|^2<\infty.
\end{align}

Let us first show some relation between the Sobolev-type space $W_{\ba,\bb}^l(\Rd)$ (see \eqref{Sobolev-type space}) and the normal Sobolev space $H^l(\Rd)$.
\begin{lemma}\label{lemma}
	For $u\in\W_{\ba,\bb}^{|\bk|_1+|\br|_1}(\Rd)$, for any $\br,\bk\in\Nd$, we have
\begin{align*}	
	\left|\left|\bx^{\br}\bp_{\bx}^{\bk}u\right|\right|\lesssim\left(\prod_{i=1}^d\a_i^{-r_i}\right)|\bk+\br|_{\textup{mix}}^{\frac12}\cdot||u||_{\W^{|\bk|_1+|\br|_1}_{\ba,\bb}(\Rd)}.
\end{align*}
\end{lemma}
\begin{proof}
For clarity, we show it holds for $d=1$ in detail.
\begin{align}\label{lemma3.1_eq0}
	\left|\left|x^r\p_x^ku\right|\right|^2
		=\left|\left|\sum_{n=0}^\infty\hat{u}_n^{\a,\b}x^r\p_x^k\H_n^{\a,\b}(x)\right|\right|^2
		\overset{\eqref{recurrence for RT hermite function},\eqref{derivative}}=\a^{-2r}\left|\left|\sum_{n=0}^\infty\hat{u}_n^{\a,\b}\sum_{i=-(k+r)}^{k+r}\eta_{n,i}\H_{n+i}^{\a,\b}(x)\right|\right|^2,
\end{align}
where, for each $n$, $\eta_{n,i}$ is a product of $k+r$ factors of $\left(\pm\frac{\sqrt{\lambda_{n+i}}}2\right)$ or $\frac\b2$ with $-(k+r)\leq i\leq k+r$. Notice that 
\begin{align}\label{lemma3.1_eq3}
	\lambda_{n+i}\sim\lambda_{n+j},
\end{align}
provided that $\lambda_{n+i},\lambda_{n+j}\neq0$, for all $-(k+r)\leq i,j\leq k+r$.
In fact, it is equivalent to show that $\lambda_n\sim\lambda_{n+l}$, for all $0\leq l\leq 2(k+r)$. By convention, $\lambda_n=0$, if $n\leq0$. Notice that
\begin{align*}
	\frac{\lambda_n}{\lambda_{n+l}}=\frac n{n+l}\leq1\quad
\textup{and}\quad \frac n{n+l}\geq\frac1{1+l}\geq\frac1{1+2(k+r)}, \ \forall n\geq1.
\end{align*}
Meanwhile $\lim_{n\rightarrow\infty}\frac n{n+l}=1$, for all $0\leq l\leq2(k+r)$. Therefore, $\frac{\lambda_n}{\lambda_{n+l}}\sim1$. Hence, $\eta_{n,j}\lesssim\sqrt{\mu_{n,k+r}}$, by \eqref{mu part}, \eqref{lemma3.1_eq3}. Thus, 
\begin{align}\label{lemma3.1_eq1}\notag
	\left|\left|x^r\p_x^ku\right|\right|^2\overset{\eqref{lemma3.1_eq0}}\sim&\a^{-2r}\left|\left|\sum_{n=0}^\infty\hat{u}_n^{\a,\b}\sqrt{\mu_{n,k+r}}\sum_{i=-(k+r)}^{k+r}\H_{n+i}^{\a,\b}(x)\right|\right|^2\\
		=&\a^{-2r}\sum_{n=0}^\infty\hat{u}_n^{\a,\b}\sqrt{\mu_{n,k+r}}\sum_{i=-(k+r)}^{k+r}\sum_{l=0}^\infty\hat{u}_l^{\a,\b}\sqrt{\mu_{l,k+r}}\left\langle\H_{n+i}^{\a,\b}(x),\sum_{j=-(k+r)}^{k+r}\H_{l+j}^{\a,\b}(x)\right\rangle.
\end{align}
It is clear that the scalar product in \eqref{lemma3.1_eq1} is nonzero only if $l=n+i-j$. And $\mu_{n,k+r}\sim\mu_{n+i-j,k+r}$, for all $-(k+r)\leq i,j\leq k+r$. It can be verified by \eqref{mu part} and \eqref{lemma3.1_eq3}. Therefore,
\begin{align*}
	\left|\left|x^r\p_x^ku\right|\right|^2		\overset{\eqref{lemma3.1_eq1}}\sim&\a^{-2r}\sum_{n=0}^\infty\mu_{n,k+r}\hat{u}_n^{\a,\b}\sum_{\tilde{l}=-2(k+r)}^{2(k+r)}\hat{u}_{n+\tilde{l}}^{\a,\b}
	\leq\a^{-2r}\sum_{n=0}^\infty\mu_{n,k+r}\sum_{\tilde{l}=-2(k+r)}^{2(k+r)}\left|\hat{u}_n^{\a,\b}\right|\left|\hat{u}_{n+\tilde{l}}^{\a,\b}\right|\\
\leq&\a^{-2r}\sum_{n=0}^\infty\mu_{n,k+r}\frac12\sum_{\tilde{l}=-2(k+r)}^{2(k+r)}\left(\left|\hat{u}_n^{\a,\b}\right|^2+\left|\hat{u}_{n+\tilde{l}}^{\a,\b}\right|^2\right)\\
	=&\a^{-2r}\sum_{n=0}^\infty\mu_{n,k+r}\left[2(k+r)\left|\hat{u}_n^{\a,\b}\right|^2+\frac12\sum_{\tilde{l}=-2(k+r)}^{2(k+r)}\left|\hat{u}_{n+\tilde{l}}^{\a,\b}\right|^2\right]\\
	=&2(k+r)\a^{-2r}\sum_{n=0}^\infty\mu_{n,k+r}\left|\hat{u}_n^{\a,\b}\right|^2+\frac12\a^{-2r}\sum_{\tilde{n}=0}^\infty\sum_{\tilde{l}=-2(k+r)}^{2(k+r)}\mu_{\tilde{n}-\tilde{l},k+r}\left|\hat{u}_{\tilde{n}}^{\a,\b}\right|^2\\
	\sim&\a^{-2r}4(k+r)\sum_{n=0}^\infty\mu_{n,k+r}\left|\hat{u}_n^{\a,\b}\right|^2
	\lesssim\a^{-2r}(k+r)||u||^2_{\W_{\a,\b}^{k+r}(\R)}.
\end{align*}
Till now, we have shown that \eqref{lemma3.1_eq2} holds for $d=1$. For $d\geq2$, we shall proceed the argument similarly as for $d=1$.
\begin{align*}
	\left|\left|\bx^{\br}\bp_{\bx}^{\bk}u\right|\right|^2
		=&\left(\prod_{\tilde{i}=1}^d\a_{\tilde{i}}^{-2r_{\tilde{i}}}\right)\left|\left|\sum_{\bn\in\Nd}\hat{u}_{\bn}^{\ba,\bb}\sum_{-(\bk+\br)\leq\bi\leq\bk+\br}\bm{\eta}_{\bn,\bi}\bH_{\bn+\bi}^{\ba,\bb}(\bx)\right|\right|^2\\
		\sim&\left(\prod_{\tilde{i}=1}^d\a_{\tilde{i}}^{-2r_{\tilde{i}}}\right)\left|\left|\sum_{\bn\in\Nd}\hat{u}_{\bn}^{\ba,\bb}\sqrt{\bm{\mu}_{\bn,\bk+\br}}\sum_{-(\bk+\br)\leq\bi\leq(\bk+\br)}\bH_{\bn+\bi}^{\ba,\bb}(\bx)\right|\right|^2\\
\lesssim&\left(\prod_{\tilde{i}=1}^d\a_{\tilde{i}}^{-2r_{\tilde{i}}}\right)\sum_{\bn\in\Nd}\bm{\mu}_{\bn,\bk+\br}\sum_{-2(\bk+\br)\leq\tilde{\bl}\leq2(\bk+\br)}\left(\left|\hat{u}_{\bn}^{\ba,\bb}\right|^2+\left|\hat{u}_{\bn+\tilde{\bl}}^{\ba,\bb}\right|^2\right)\\
		\sim&\left(\prod_{\tilde{i}=1}^d\a_{\tilde{i}}^{-2r_{\tilde{i}}}\right)|\bk+\br|_{\textup{mix}}\sum_{\bn\in\Nd}\bm{\mu}_{\bn,\bk+\br}\left|\hat{u}_{\bn}^{\ba,\bb}\right|^2\\
	\lesssim&\left(\prod_{\tilde{i}=1}^d\a_{\tilde{i}}^{-2r_{\tilde{i}}}\right)|\bk+\br|_{\textup{mix}}\cdot||u||^2_{\W_{\ba,\bb}^{|\bk|_1+|\br|_1}(\Rd)}.
\end{align*} 
Therefore, we obtain the desired result.
\end{proof}

\begin{corollary}\label{coro}
	For $u\in\W_{\ba,\bb}^m(\Rd)$, we have
	$||u||_{H^m(\Rd)}\lesssim||u||_{\W^m_{\ba,\bb}(\Rd)}$, for all $m\geq0$.
\end{corollary}
\begin{proof}
	Compared the definitions of $W_{\ba,\bb}^m(\Rd)$ and $H^m(\Rd)$ in \eqref{sobolev norm} and \eqref{H-space}, it remains to show that 
\begin{align}\label{lemma3.1_eq2}
	\left|\left|\bp_{\bx}^{\bk}u\right|\right|^2\lesssim\left|\left|\bD_{\bx}^{\bk}u\right|\right|^2\overset{\eqref{norm in frequency}}=\sum_{\bn\in\Nd}\bm{\mu}_{\bn,\bk}\left|\hat{u}_{\bn}^{\ba,\bb}\right|^2,
\end{align}
for all $0\leq|\bk|_1\leq m$. The desired result is followed immediately from Lemma \ref{lemma} by letting $\br=\bm{0}$, i.e., 
\begin{align*}
	\left|\left|\bp_{\bx}^{\bk}u\right|\right|^2\lesssim|\bk|_{\textup{mix}}\cdot\left|\left|\bD_{\bx}^{\bk}u\right|\right|^2.
\end{align*}
\end{proof}
The convergence rate of the HSM with the HC approximation under the assumptions $\bm{(C_1)}$-$\bm{(C_3)}$ is:
\begin{theorem}\label{thm1}
	Assume that conditions $\bm{(C_1)}$-$\bm{(C_3)}$ are satisfied, and the solution $u\in L^\infty(0,T;\K_{\ba,\bb}^m(\Rd))\cap L^2(0,T;\K_{\ba,\bb}^m(\Rd))$, for $m>1$. Let $u_N$ be the approximate solution obtained by HSM \eqref{HSM}, then
	\begin{align*}
		||u-u_N||(t)\lesssim c^*N^{\frac{1-m}2},
	\end{align*}
where $c^*$ depends on $\ba$, the norms of $L^2(0,T;\K_{\ba,\bb}^m(\Rd))$ and  $L^\infty(0,T;\K_{\ba,\bb}^m(\Rd))$.
\end{theorem}
\begin{proof}
	For the notational convenience, we denote $U_N=P_N^{\ba,\bb}u$. It is readily verified that
\begin{align}\label{thm3.1_eq1}
	\langle\p_t(u-U_N),\varphi\rangle=0\quad
	\Rightarrow\quad\langle\p_tU_N,\varphi\rangle=\langle-Lu+f,\varphi\rangle,\quad\forall\, \varphi\in X_N^{\ba,\bb}.
\end{align}
Combined with the formulation of Hermite spectral method \eqref{HSM}, we have
\begin{align*}
	\langle\p_t(U_N-u_N),\varphi\rangle
		=&\langle-Lu+f,\varphi\rangle+\A(u_N,\varphi)+\langle f,\varphi\rangle
		=\A(u_N-u,\varphi)\\
		=&-\A(u-U_N,\varphi)-\A(U_N-u_N,\varphi),\quad\forall\,\varphi\in X_N^{\ba,\bb}.
\end{align*}
Take $\varphi=2(U_N-u_N)\in X_N^{\ba,\bb}$, then
\begin{align*}
	\p_t||U_N-u_N||^2
		=&-2\A(u-U_N,U_N-u_N)-2\A(U_N-u_N,U_N-u_N)\\
		\overset{\eqref{continuous},\eqref{coercive}}\leq&2C||u-U_N||_{H^1_0(\Rd)}||U_N-u_N||_{H^1_0(\Rd)}-2c||U_N-u_N||_{H_0^1(\Rd)}^2\\
		\lesssim&||u-U_N||_{H^1_0(\Rd)}^2,\quad\textup{by\ Young's\ inequality}.
\end{align*}
With Corollary \ref{coro} and Corollary \ref{coro2.2} ( if OHC approximation is considered), we have
\begin{align*}
	&\p_t||U_N-u_N||^2\lesssim||u-U_N||^2_{\W_{\ba,\bb}^1(\Rd)}
		\lesssim N^{1-m}|u|^2_{\K_{\ba,\bb}^m(\Rd)}\\
	\Rightarrow\quad&||U_N-u_N||^2(t)\lesssim N^{\frac{1-m}2}\left[\int_0^t|u|_{\K_{\ba,\bb}^m(\Rd)}^2(s)ds\right]^{\frac12}.
\end{align*}
The same estimate holds for RHC approximation with Corollary \ref{coro2.2} replaced by Corollary \ref{coro2.1}. And then, it yields that 
\begin{align*}
	||u-u_N||(t)\leq&||u-U_N||(t)+||U_N-u_N||(t)\\
		\lesssim&N^{-\frac m2}|u|_{\K_{\ba,\bb}^m(\Rd)}(t)+N^{\frac{1-m}2}\left[\int_0^t|u|_{\K_{\ba,\bb}^m(\Rd)}^2(s)ds\right]^{\frac12}
		\lesssim c^*N^{\frac{1-m}2},
\end{align*}
where $c^*$ depends on $\ba$, the norms of $L^2(0,T;\K_{\ba,\bb}^m(\Rd))$ and  $L^\infty(0,T;\K_{\ba,\bb}^m(\Rd))$.
\end{proof}

However, the assumptions $\bm{(C_1)}$ and $\bm{(C_2)}$ are not easy to verify. In the sequel, we make assumptions on the operator $L$ and the convergence rate of the HSM is investigated under the conditions below. Assume that
\begin{enumerate}
	\item[$\bm{(C_4)}$] The operator $L$ (c.f. \eqref{L}) is strongly elliptic and uniformly bounded, i.e.,
\begin{align*}
	\sum_{i,j=1}^da_{ij}(\bx)\xi_i\xi_j\geq\theta|\xi|^2,\quad\forall \xi\in\Rd,\quad
	\textup{and}\quad	||\bm{A}||_\infty=\max_{i,j=1,\cdots,d}{||a_{ij}||_\infty}<\infty,
\end{align*}
for $\bx\in\Rd$, where $\theta>0$.
	\item[$\bm{(C_5)}$] There exists some constant $C>0$, such that
\begin{align*}
	c(\bx)-\frac12 \nabla\cdot \bm{b}(\bx)\geq-C,
\end{align*}
for all $\bx\in\Rd$.
	\item[$\bm{(C_6)}$] There exist some integer indices  $\bg,\bd\in\Nd$, such that
\begin{align*}
	c(\bx)\lesssim 1+\bx^{2\bg}\quad\textup{and}\quad
	b_i(\bx)\lesssim 1+\bx^{2\bd},\quad\forall\, i=1,2,\cdots,d,
\end{align*}
for all $\bx\in\Rd$.
\end{enumerate}
\begin{theorem}
	Assume that conditions $\bm{(C_3)}$-$\bm{(C_6)}$ are satisfied and the solution to the equation \eqref{PDE} $u\in L^2(0,T;\K_{\ba,\bb}^m(\Rd))$, for some integer $m>\max\{|\bg|_1,|\bd|_1+1\}$, and let $u_N$ be the approximate solution obtained by HSM \eqref{HSM}, then 
\begin{align*}
	||u-u_N||(t)\lesssim c^{\sharp}N^{\frac{\max\{|\bg|_1,|\bd|_1+1\}-m}2},
\end{align*}
where $c^{\sharp}$ depends on $\ba$, $T$ and the norm of $L^2(0,T;\K_{\ba,\bb}^m(\Rd))$.
\end{theorem}
\begin{proof}
	Similarly as we did in the proof of Theorem \ref{thm1}, denote $U_N=P_N^{\ba,\bb}u$ for convenience, and let $\varphi=2(U_N-u_N)\in X_N^{\ba,\bb}$, then
\begin{align}\label{thm2_eq0}
	\p_t||U_N-u_N||^2
		=-2\A(u-U_N,U_N-u_N)-2\A(U_N-u_N,U_N-u_N):=\rom{5}_1+\rom{5}_2,
\end{align}
where $\A$ is defined in \eqref{bilinear form}. For $\rom{5}_2$, 
\begin{align}\label{V2}\notag
	-\frac12\rom{5}_2
	&=\int_{\Rd}(\nabla(U_N-u_N))^T\bm{A}(\nabla(U_N-u_N))
		+\int_{\Rd}(U_N-u_N)\bm{b}\cdot\nabla(U_N-u_N)
		+\int_{\Rd}c(U_N-u_N)^2\\\notag
	&=\int_{\Rd}(\nabla(U_N-u_N))^T\bm{A}(\nabla(U_N-u_N))
		+\int_{\Rd}\left(c-\frac12\nabla\cdot\bm{b}\right)(U_N-u_N)^2\\
	&\overset{\bm{(C_4)}, \bm{(C_5)}}\geq\theta||\nabla(U_N-u_N)||^2-C||U_N-u_N||^2.
\end{align}
Meanwhile for $\rom{5}_1$, 
\begin{align}\label{V1}\notag
	|\rom{5}_1|
=&2\left[\int_{\Rd}(\nabla(u-U_N))^T\bm{A}(\nabla(U_N-u_N))
	+\int_{\Rd}(U_N-u_N)\bm{b}\cdot\nabla(u-U_N)\right.\\\notag
	&\phantom{\frac12[}\left.+\int_{\Rd}c(u-U_N)(U_N-u_N)\right]\\\notag
\leq&2[||\bm{A}||_\infty||\nabla(u-U_N)||\cdot||\nabla(U_N-u_N)||+||\bm{b}\cdot\nabla(u-U_N)||\cdot||U_N-u_N||\\\notag
	&\phantom{\frac12[}+||c(u-U_N)||\cdot||U_N-u_N||]\\\notag
	\lesssim&C_{||\bm{A}||_\infty,\theta}||\nabla(u-U_N)||^2+2\theta||\nabla(U_N-u_N)||^2+||\bm{b}\cdot\nabla(u-U_N)||^2+||c(u-U_N)||^2\\
	&+||U_N-u_N||^2.
\end{align}
On the right-hand side of \eqref{V1}, the third and forth terms are to be estimated.
\begin{align}\label{thm2_eq1}\notag
	||c(u-U_N)||^2
	&\overset{\bm{(C_6)}}\lesssim||(1+\bx^{2\bg})(u-U_N)||^2
	\lesssim||u-U_N||^2+||\bx^{2\bg}(u-U_N)||^2\\
	&\lesssim ||u-U_N||^2+\left(\prod_{i=1}^d\a_i^{-4\g_i}\right)|\bg|_{\textup{mix}}\cdot||u-U_N||^2_{\W_{\ba,\bb}^{|\bg|_1}(\Rd)},
\end{align}
by Lemma \ref{lemma}. Similarly, from $\bm{(C_6)}$ again, we deduce that 
\begin{align}\label{thm2_eq2}\notag
	\left|\left|\bm{b}\cdot\nabla(u-U_N)\right|\right|^2	
		\leq&\sum_{i=1}^d\left|\left|b_i(\bx)\p_{x_i}(u-U_N)\right|\right|^2
		\lesssim\sum_{i=1}^d\left|\left|(1+\bx^{2\bd})\p_{x_i}(u-U_N)\right|\right|^2\\\notag
	\leq&\sum_{i=1}^d\left|\left|\p_{x_i}(u-U_N)\right|\right|^2+\sum_{i=1}^d\left|\left|\bx^{2\bd}\p_{x_i}(u-U_N)\right|\right|^2\\\notag
	\lesssim&||u-U_N||^2_{\W_{\ba,\bb}^1(\Rd)}+\sum_{i=1}^d\left(\prod_{i=1}^d\a_i^{-4\delta_i}\right)|\bd+\bm{e}_i|_{\textup{mix}}\cdot||u-U_N||^2_{\W_{\ba,\bb}^{|\bd|_1+1}(\Rd)}\\
	\lesssim&||u-U_N||^2_{\W_{\ba,\bb}^1(\Rd)}+d\left(\prod_{i=1}^d\a_i^{-4\delta_i}\right)|\bd+1|_{\textup{mix}}\cdot||u-U_N||^2_{\W_{\ba,\bb}^{|\bd|_1+1}(\Rd)}.
\end{align}
Combine \eqref{thm2_eq0}-\eqref{V1}, we have
\begin{align*}
	\p_t||u_N-U_N||^2&\lesssim||\nabla(u-U_N)||^2+||\bm{b}\cdot\nabla(u-U_N)||^2+||c(u-U_N)||^2+C||u_N-U_N||^2\\
	\overset{\eqref{thm2_eq1},\eqref{thm2_eq2}}\lesssim&||\nabla(u-U_N)||^2+C||u_N-U_N||^2+||u-U_N||^2_{\W_{\ba,\bb}^1(\Rd)}\\
	&+||u-U_N||^2_{\W_{\ba,\bb}^{|\bd|_1+1}(\Rd)}+||u-U_N||^2_{\W_{\ba,\bb}^{|\bg|_1}(\Rd)}\\
	\lesssim& C||u_N-U_N||^2+N^{\max\{|\bg|_1,|\bd|_1+1\}-m}|u|^2_{\K_{\ba,\bb}^m(\Rd)},
\end{align*}
by Corollary \ref{coro2.1} or Corollary \ref{coro2.2}. Hence,
\begin{align*}
	||u_N-U_N||^2(t)\leq& e^{Ct}||u_N-U_N||^2(0)+N^{\max\{|\bg|_1,|\bd|_1+1\}-m}e^{Ct}\int_0^te^{-Cs}|u|^2_{\K_{\ba,\bb}^m(\Rd)}(s)ds\\
	\leq& N^{\max\{|\bg|_1,|\bd|_1+1\}-m}\int_0^te^{C(t-s)}|u|^2_{\K_{\ba,\bb}^m(\Rd)}(s)ds.
\end{align*}
Therefore,
\begin{align*}
	||u-u_N||^2(t)\leq&||u-U_N||^2(t)+||u_N-U_N||^2(t)\\
		\lesssim&N^{1-m}|u|^2_{\K_{\ba,\bb}^m(\Rd)}(t)+N^{\max\{|\bg|_1,|\bd|_1+1\}-m}\int_0^te^{C(t-s)}|u|^2_{\K_{\ba,\bb}^m(\Rd)}(s)ds\\
		\lesssim&N^{\max\{|\bg|_1,|\bd|_1+1\}-m}\int_0^T|u|^2_{\K_{\ba,\bb}^m(\Rd)}(s)ds.
\end{align*}
The desired result is obtained.
\end{proof}

\section{Numerical results}

\setcounter{equation}{0}

\subsection{HC approximations with Hermite functions}

In Figure \ref{fig-d2_index}, we display the indices of RHC and OHC (with $\gamma=0.5$) in dimension 2 with $N=31$. It is clear to see that the indices of OHC is a subset of RHC. Furthermore, we list the number of indices of $N=31$ with dimension ranging from 2 to 5.
\smallskip
\begin{center}
\begin{tabular}{|c|c|c|c|c|}
\hline	
	dim&2&3&4&5\\
\hline	
	\# of indices in RHC&176&712&2485&7922\\
\hline
	\# of indices in OHC ($\gamma=0.5$)&136&440&1264&3392\\
\hline
\end{tabular}
\end{center}
\smallskip
\begin{figure}
  {\centering \includegraphics[trim = 15mm 90mm 0mm 90mm, clip, scale=0.8]{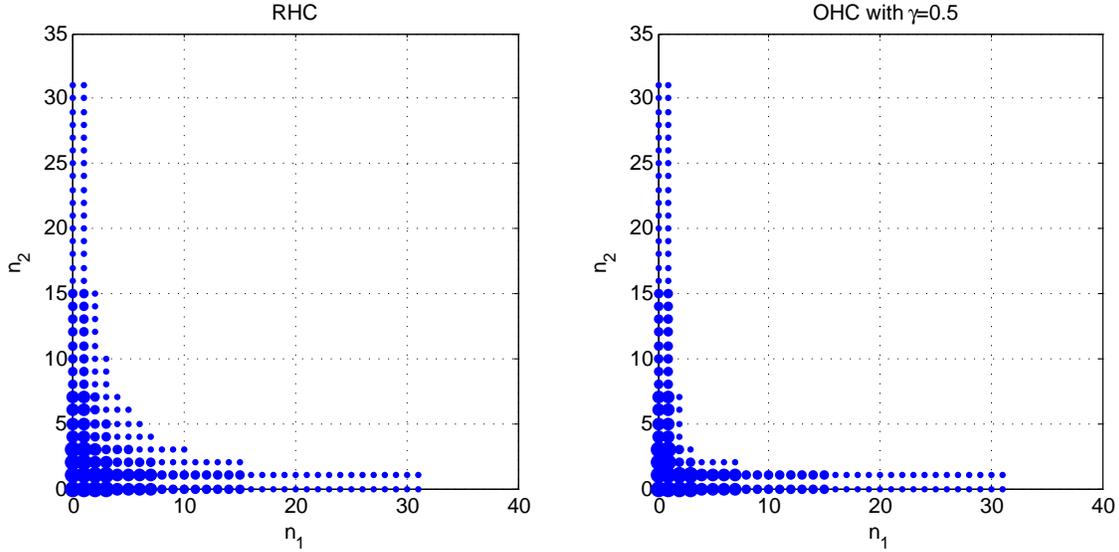}}
	\caption{For $d=2$, $N=31$. Left: the index set $\bO_N$ of RHC. Right: the index set $\bO_{N,\g}$ of OHC with $\gamma=0.5$. }
\label{fig-d2_index}
\end{figure}

It is well-known that the abscissas of Hermite polynomials are non-nested, except the origin. It will lead more number of points than those nested quadrature, such as Chebyshev polynomials. However, the number is still dramatically reduced, compared to the full grids. We list the abscissas of RHC, OHC and full grid of $N=31$ with the dimension ranging from $2$ to $4$.
\smallskip
\begin{center}
\begin{tabular}{|c|c|c|c|}
\hline
	dim&2&3&4\\
\hline
	\# of abscissas in OHC ($\g=0.5$)&108&3348&28944\\ 
\hline
	\# of abscissas in RHC&	298&	6612&82704\\
\hline
	\# of abscissas in full grid&	961&	29791&	923521\\
\hline
\end{tabular}
\end{center}
\smallskip
It is clear that the abscissa in RHC/OHC is much fewer than thoses in the full grid.

\subsection{HSM with sparse grid}

Although the HC approximation is theoretically feasible, it is not suitable for practical implementations, due to the unclarity ``combining effecting" of the product rules, i.e. how to determine the weights from different combinations of 1-D Gauss-Hermite quadrature. Thus, in this subsection, we stick to the Smolyak's algorithm \cite{S} to test the accuracy of high-dimensional HSM applying to linear parabolic PDE.

Let us recall that the Smolyak's algorithm is given
\begin{align*}
	\mathcal{I}(L,d)=\sum_{L-d+1\leq|\bi|_1\leq L}(-1)^{L-|\bi|_1}
	\dbinom{d-1}{L-|\bi|_1} 
	(\mathcal{U}^{i_1}\otimes\cdots\otimes\mathcal{U}^{i_d}),
\end{align*}
where $\mathcal{U}^i$ is an indexed family of 1D quadrature, $\bm{i}$ is the 1D level, $\bi=(i_1,\cdots,i_d)$ is the level vector, $L$ is the max level. The sparse grid is formed by weighted combinations of those product rules whose product level $|\bi|_1$ falls between $L-d+1$ and $L$.

\begin{figure}
  {\centering \includegraphics[trim=0mm 90mm 0mm 90mm, clip, scale=0.7]{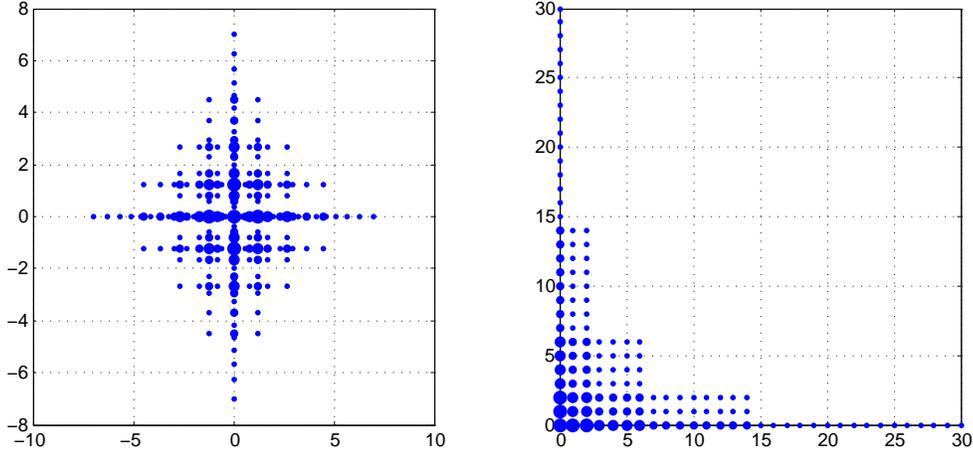}}
	\caption{ In $d=2$, level $L$ ranging from 2 to 4. Left: the abscissas of Hermite functions.  Right: the indices in the index set. The larger the dot is, the lower level it belongs to.}
\label{fig-d2_grid}
\end{figure}

In Figure \ref{fig-d2_grid}, we display the abscissas of the Hermite functions and the index set with level $L$ ranging from $2$ to $4$ in $d=2$.

Let us test the accuracy with the following linear parabolic PDE
\begin{align*}
	\left\{\begin{aligned}
		\p_tu&=\triangle u-\sum_{i=1}^d x_i^2u+f(\bx,t)\\
		u(\bx,0)&=\left(\sum_{i=1}^dx_i\right)e^{-\frac12(x_1^2+\cdots+x_d^2)}
	\end{aligned}\right.,
\end{align*}
where $\triangle$ is the Laplacian operator, 
\begin{align*}
	f(\bx,t)=\left[\cos{t}+d\sin{t}+(d+2)\sum_{i=1}^dx_i \right]e^{-\frac12(x_1^2+\cdots x_d^2)}.
\end{align*}
By direct computations, the exact solution to this PDE is 
\begin{align*}	u(\bx,t)=\left(\sum_{i=1}^dx_i+\sin{t}\right)e^{-\frac12(x_1^2+\cdots+x_d^2)}.
\end{align*}
It is known from \cite{LY} that the best scaling factor is $\ba=\bm{1}$ in this case, since the first two Hermite functions will resolve the exact solution perfectly only with the round-off errors (around $10^{-16}$ on my computer). To make the convergence rate observable with respect to the level $L$, we shall choose the scaling factor $\ba$ to be $1.01\times\bm{1}$. 

The corresponding spectral scheme (cf. \eqref{HSM}, \eqref{bilinear form}) is as follows:
\begin{align}\label{example_HSM}
	\left\{\begin{aligned}
		\langle\p_tu_N(t),\varphi\rangle&=-\langle\nabla u_N,\nabla\varphi\rangle-\sum_{i=1}^d\langle x_i^2u_N,\varphi\rangle+\langle f,\varphi\rangle\\
		u_N(0)&=P_Nu_0,
	\end{aligned}\right.
\end{align}
for all $\varphi\in X_N$. Here, we choose $X_N=X_N^{\a,\b}=\textup{span}\{\H_{\bn}^{\ba,\bb}:\, \bO_N\ \textup{from\ Smolyak}\}$. Thus, we can write the numerical solution as
\begin{align*}
	u_N(\bx,t)=\sum_{\bn\in\bO_N}a_{\bn}(t)\H_{\bn}^{\ba,\bb}(\bx).
\end{align*}
Taking $\varphi(\bx)=\H_{\bn}^{\ba,\bb}(\bx)$ in \eqref{example_HSM}. Due to \eqref{derivative}, \eqref{recurrence for RT hermite function} and \eqref{Hermite representation}, we arrive at an ODE
\begin{align}\label{ODE}
	\left\{\begin{aligned}
		\p_ta_{\bn} &= Aa_{\bn}+\hat{f}_{\bn}\\
		a_{\bn}(0)&=\left(\hat{u}_0\right)_{\bn},
	\end{aligned}\right.
\end{align} 
where $\hat{f}_{\bn}$ (resp. $\left(\hat{u}_0\right)_{\bn}$) is the Hermite coefficients of $f$ (resp. $u_0$) and the matrix $A$ comes from the Laplacian operator and the potential. We display the nonzero entries of the matrix $A$ for dimension $3$ and $4$ with level$=4$ in Figure \ref{fig-A_d4_l4}.

\begin{figure}
		{\centering \includegraphics[trim = 35mm 105mm 35mm 105mm, clip, scale=1.1]{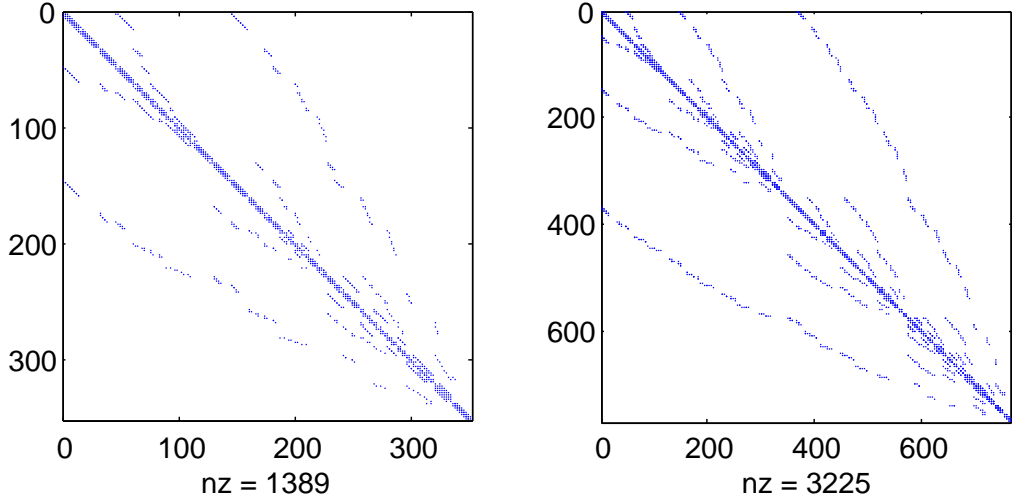}}
	\caption{The nonzero entries in the matrix $A$ (cf. \eqref{ODE}) are displayed with level$=4$. Left: $d=3$, Right: $d=4$.}
\label{fig-A_d4_l4}
\end{figure}

We adopt the central difference scheme to solve \eqref{ODE} with $T=0.1$, $dt=10^{-5}$, $\ba=1.01\times\bm{1}$ and $\bb=\bm{0}$. Figure \ref{fig-l2error} shows the $L^2-$norm of $\left(u_N-u_{\textup{exact}}\right)$ with respect to the level in dimension ranging from $2$ to $4$. It is exactly what we expect that in the semi-log plot the error goes down almost along a straight line, which indicates that the convergence rate is nearly exponential decaying. However, with the dimension grows, the error becomes slightly larger. It reveals that the convergence rate still slightly deteriorates with the dimension increasing.

\begin{figure}[htb]
  \centering
  \begin{minipage}[c]{0.38\textwidth}
    \centering
    \begin{tabular}{|c|c|c|c|}
\hline	
	level/dim&	2&	3&	4\\
\hline
	2&	2.24E-03&7.99E-03&	n/a\\
\hline
	3&3.99E-04&544E-03&2.10E-02\\
\hline
	4&	4.75E-06&1.93E-03&1.14E-02\\
\hline
	5&	2.72E-07&2.66E-04&4.11E-03\\
\hline
\end{tabular}
  \end{minipage}
  \begin{minipage}[c]{0.58\textwidth}
    \includegraphics[trim = 10mm 80mm 10mm 90mm, clip, scale=.55]{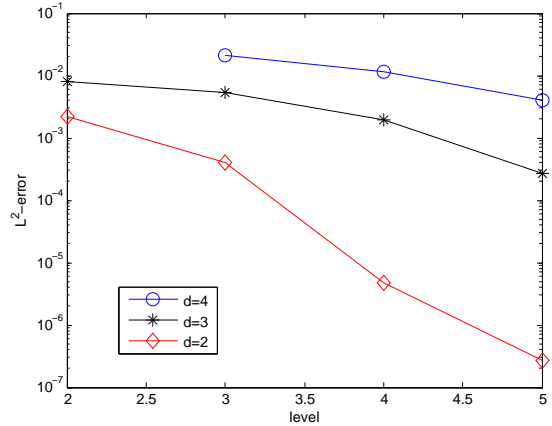}
  \end{minipage}
\caption{The $L^2$ error of $u_N$ with respect to the level in $d=2,3$ and $4$ is drawn.}
\label{fig-l2error}
 \end{figure}

\section{Conclusion}

In this paper, we consider the HC approximation with generalized Hermite functions. We established the error estimate in the appropriate space for both RHC and OHC. Furthermore, the error estimate of the dimensional adaptive approximation is obtained with respect to the dependence of dimension. As an application, the HC approximation is applied to high-dimensional linear parabolic PDEs.  We investigated the convergence rate of the Galerkin-type HSM in the suitable weighted Korobov space. It is shown to be exponential convergent. Moreover, the numerical simulation supports our theoretical proofs.

\appendix
\numberwithin{equation}{section}
\numberwithin{theorem}{section}
 
\section{Proof of Theorem \ref{projection on full grid}}

\begin{proof}[Proof of Theorem \ref{projection on full grid}]
Let $\bO_N^c=\{\bn\in\Nd:\ |\bn|_\infty>N\}$. By \eqref{PN in frequency}, \eqref{sobolev seminorm} and \eqref{seminorm in frequency},
\begin{align}\label{thm2.1_eq1}
	\left|P_N^{\ba,\bb}u-u\right|^2_{\W_{\ba,\bb}^l(\Rd)}
		=\sum_{j=1}^d\sum_{\bn\in\bO_N^c}\mu_{n_j,l}\left|\hat{u}_{\bn}^{\ba,\bb}\right|^2.
\end{align}
For any $1\leq j\leq d$, 
\begin{align}\label{thm2.1_eq2}
	\sum_{\bn\in\bO_N^c}\mu_{n_j,l}\left|\hat{u}_{\bn}^{\ba,\bb}\right|^2
		=\sum_{\bn\in\bm{\Lambda}_N^{1,j}}\mu_{n_j,l}\left|\hat{u}_{\bn}^{\ba,\bb}\right|^2
			+\sum_{\bn\in\bm{\Lambda}_N^{2,j}}\mu_{n_j,l}\left|\hat{u}_{\bn}^{\ba,\bb}\right|^2:=\rom{1}_1+\rom{1}_2,
\end{align}
where $\bm{\Lambda}_N^{1,j}=\{\bn\in\bO_N^c:\ n_j>N\}$ and $\bm{\Lambda}_N^{2,j}=\{\bn\in\bO_N^c:\ n_j\leq N\}$. For $\rom{1}_1$:
\begin{align}\label{thm2.1_I1}
	\rom{1}_1\leq\max_{\bn\in\bm{\Lambda}_N^{1,j}}\left\{\frac{\mu_{n_j,l}}{\mu_{n_j,m}}\right\}\sum_{\bn\in\bm{\Lambda}_N^{1,j}}\mu_{n_j,m}\left|\hat{u}_{\bn}^{\ba,\bb}\right|^2
		\lesssim |\ba|_\infty^{2(l-m)}N^{l-m}|u|^2_{\W_{\ba,\bb}^m(\Rd)}.
\end{align}
In fact,
\begin{align*}
	\max_{\bn\in\bm{\Lambda}_N^{1,j}}\left\{\frac{\mu_{n_j,l}}{\mu_{n_j,m}}\right\}
		=&\max_{\bn\in\bm{\Lambda}_N^{1,j}}\left\{\frac{2^{l-m}\a_j^{2(l-m)}}{(n_j-l)(n_j-l-1)\cdots(n_j-m+1)}\right\}\\
		\leq&2^{l-m}|\ba|_\infty^{2(l-m)}(N-m+1)^{l-m}.
\end{align*}
For $\rom{1}_2$, if $\bn\in\bm{\Lambda}_N^{2,j}$, then there exists some $k\neq j$, such that $n_k>N$.
\begin{align}\label{thm2.1_I2}
	\rom{1}_2
		\leq\max_{\bn\in\bm{\Lambda}_N^{2,j}}\left\{\frac{\mu_{n_j,l}}{\mu_{n_k,m}}\right\}\sum_{\bn\in\bm{\Lambda}_N^{2,j}}\mu_{n_k,m}\left|\hat{u}_{\bn}^{\ba,\bb}\right|^2
		\lesssim|\ba|_\infty^{2l}\left|\frac1{\ba}\right|_\infty^{2m}N^{l-m-2}|u|^2_{\W_{\ba,\bb}^m(\Rd)},
\end{align}
since
\begin{align*}
	\max_{\bn\in\bm{\Lambda}_N^{2,j}}\left\{\frac{\mu_{n_j,l}}{\mu_{n_k,m}}\right\}
		=&\max_{\bn\in\bm{\Lambda}_N^{2,j}}\left\{2^{l-m}\frac{\a_j^{2l}}{\a_k^{2m}}\frac{\frac{n_j!}{(n_j-l)!}}{\frac{n_k!}{(n_k-m)!}}\right\}
		\leq2^{l-m}|\ba|_\infty^{2l}\left|\frac1{\ba}\right|_\infty^{2m}\frac{\frac{N!}{(N-l)!}}{\frac{(N+1)!}{(N+1-m)!}}\\
		=&2^{l-m}|\ba|_\infty^{2l}\left|\frac1{\ba}\right|_\infty^{2m}\frac1{N+1}\frac1{(N-l)(N-l-1)\cdots(N-m)}\\
		\leq&2^{l-m}|\ba|_\infty^{2l}\left|\frac1{\ba}\right|_\infty^{2m}(N-m)^{l-m-2}.
\end{align*}
Combine \eqref{thm2.1_eq1}-\eqref{thm2.1_I2}, we obtain the result. Furthermore, the mix derivatives of the order equal to or less than $m$ can be bounded by the seminorm $|u|_{\W_{\ba,\bb}^m(\Rd)}$.
\end{proof}

\section{Dimensional adaptive approximation}

The standard sparse grids are isotropic, treating all the dimensions equally. Many problems vary rapidly in only some dimensions, remaining less variable in other dimensions. In some situations, the highly changing dimensions can be recognized a priori.  Consequently it is advantageous to treat them accordingly. Without loss of generality, we assume the first $d_1$ dimensions are rapidly variable ones, and we wish to adopt the full grid. Meanwhile, the OHC approximation will be used in the rest $d_2:=d-d_1$ dimensions.

Let us denote that $\bn:=\bn_1\bigoplus\bn_2$, where $\bn_1=(n_1,\cdots,n_{d_1})$ and $\bn_2=(n_{d_1+1},\cdots,n_d)$. The index set is 
\begin{align}\label{dim adaptive indices}
	\bO_{N_1,N_2,\g}:=\left\{\bn\in\Nd:\ |\bn_1|_\infty\leq N_1,\ |\bn_2|_{\textup{mix}}|\bn_2|_\infty^{-\g}\leq N_2^{1-\g}\right\},\quad\forall\, -\infty<\g<1.
\end{align}
The complement of the index set is 
\begin{align*}
	\bO_{N_1,N_2,\g}^c:=\left\{\bn\in\Nd:\ |\bn_1|_\infty>N_1\quad\textup{or}\quad|\bn_2|_{\textup{mix}}|\bn_2|_\infty^{-\g}>N_2^{1-\g}\right\},
\end{align*}
and the $\bk-$complement of $\bO_{N_1,N_2,\g}$ is defined similarly as in \eqref{k complement of mix}:
\begin{align*}
	\bO_{N_1,N_2,\g,\bk}^c:=\left\{\bn\in\bO_{N_1,N_2,\g}^c:\ \bn\geq\bk\right\},\ \forall\,\bk\in\Nd.
\end{align*}
And the subspace $X_{N_1,N_2}^{\ba,\bb}$ is defined accordingly, i.e.,
\begin{align}\label{dim adaptive HC}
	X_{N_1,N_2}^{\ba,\bb}:=\textup{span}\{\bH_{\bn}^{\ba,\bb}(\bx):\ \bn\in\bO_{N_1,N_2,\g}\},
\end{align}
so defined the projection operator $P_{N_1,N_2,\g}^{\ba,\bb}:\ L^2(\Rd)\rightarrow X_{N_1,N_2}^{\ba,\bb}$.
\begin{theorem}\label{B.1}
	For any $u\in\K_{\ba,\bb}^m(\Rd)$, for $0<l\leq m$, we have
	\begin{align*}
		\left|P_{N_1,N_2,\g}^{\ba,\bb}u-u\right|_{W^l_{\ba,\bb}(\Rd)}\lesssim|\ba|_\infty^{l-m}\left(N_1^{l-m}+N_2^{\frac{1-\g}{d-d_1-\g}(l-m)}\right)^{\frac12}|u|_{\K_{\ba,\bb}^m(\Rd)}.
	\end{align*}
\end{theorem}
\begin{proof}
Before we proceed to prove, we divide the index set $\bO_{N_1,N_2,\g}^c$ into two subsets:
\begin{align*}
	\Gamma_{1}:=&\{\bn\in\bO_{N_1,N_2,\g}^c:\ |\bn_1|_\infty>N_1\},\\
	\Gamma_{2}:=&\{\bn\in\bO_{N_1,N_2,\g}^c:\ |\bn_1|_\infty\leq N_1\ \textup{and}\ |\bn_2|_{\textup{mix}}|\bn_2|_\infty^{-\g}>N_2^{1-\g}\}.
\end{align*}
Our proof mainly follows the proof of Theorem \ref{projection on full grid}:
\begin{align}\label{thm2.4_eq0}\notag
	\left|P_{N_1,N_2,\g}^{\ba,\bb}u-u\right|^2_{\W_{\ba,\bb}^l(\Rd)}
		\overset{ \eqref{seminorm in frequency}}=&\sum_{j=1}^d\sum_{\bn\in\bO_{N_1,N_2,\g}^c}\mu_{n_j,l}\left|\hat{u}_{\bn}^{\ba,\bb}\right|^2\\
		=&\sum_{j=1}^d\sum_{\bn\in\Gamma_1}\mu_{n_j,l}\left|\hat{u}_{\bn}^{\ba,\bb}\right|^2
		+\sum_{j=1}^d\sum_{\bn\in\Gamma_2}\mu_{n_j,l}\left|\hat{u}_{\bn}^{\ba,\bb}\right|^2		
		:=\rom{4}_1+\rom{4}_2.
\end{align}
For $\rom{4}_1$, for any $1\leq j\leq d$, 
\begin{align*}
	\rom{4}_1
=\sum_{\bn\in\Lambda_{N_1}^{1,j}}\mu_{n_j,l}\left|\hat{u}_{\bn}^{\ba,\bb}\right|^2
+\sum_{\bn\in\Lambda_{N_1}^{2,j}}\mu_{n_j,l}\left|\hat{u}_{\bn}^{\ba,\bb}\right|^2
	:=\rom{4}_{1,1}+\rom{4}_{1,2},	
\end{align*}
where 
\begin{align*}
	\Lambda_{N_1}^{1,j}:=\{\bn\in\Gamma_1:\ n_j>N_1\},\quad
	\Lambda_{N_1}^{2,j}:=\{\bn\in\Gamma_1:\ n_j\leq N_1\}.
\end{align*}
For $\rom{4}_{1,1}$:
\begin{align}\label{thm2.4_IV11}\notag
	\rom{4}_{1,1}
	\leq&\max_{\bn\in\Lambda_{N_1}^{1,j}}\left\{\frac{\mu_{n_j,l}}{\mu_{n_j,m}}\right\}\sum_{\bn\in\Lambda_{N_1}^{1,j}}\mu_{n_j,m}\left|\hat{u}_{\bn}^{\ba,\bb}\right|^2\\
	\overset{\eqref{thm2.1_I1}}\leq& 2^{l-m}|\ba|_\infty^{2(l-m)}(N_1-m+1)^{l-m}|u|_{\K_{\ba,\bb}^m(\Rd)}^2.
\end{align}
For $\rom{4}_{1,2}$, since $\bn\in\Gamma_1$, there exists some $j_0\in\{1,\cdots,d_1\}$ such that $n_{j_0}>N_1$.
\begin{align}\label{thm2.4_IV12}\notag
	\rom{4}_{1,2}
	\leq&\max_{\bn\in\Lambda_{N_1}^{2,j}}\left\{\frac{\mu_{n_j,l}}{\mu_{n_{j_0},m}}\right\}\sum_{\bn\in\Lambda_{N_1}^{2,j}}\mu_{n_{j_0},m}\left|\hat{u}_{\bn}^{\ba,\bb}\right|^2\\
	\overset{\eqref{thm2.1_I2}}\leq& 2^{l-m}|\ba|_\infty^{2l}\left|\frac1{\ba}\right|_\infty^{2m}(N_1-m)^{l-m-2}|u|_{\K_{\ba,\bb}^m(\Rd)}.
\end{align}
Hence, combine \eqref{thm2.4_IV11} and \eqref{thm2.4_IV12}, we have
\begin{align}\label{thm2.4_IV1}
	\rom{4}_1\lesssim|\ba|_\infty^{2(l-m)}N_1^{l-m}|u|_{\K_{\ba,\bb}^m(\Rd)}^2.
\end{align}
For $\rom{4}_2$, let us deduce as in \eqref{thm2.3_eq7}:
\begin{align}\label{thm2.4_eq1}
	|\bn_2|_{\textup{mix}}|\bn_2|_\infty^{-\g}>N_2^{1-\g}
	\Rightarrow |\bn_2|_\infty>N_2^{\frac{1-\g}{d-d_1-\g}}.
\end{align}
With the similar argument for $\rom{4}_1$, we write
\begin{align*}
	\rom{4}_2=\sum_{\bn\in\Lambda_{N_2}^{1,j}}\mu_{n_j,l}\left|\hat{u}_{\bn}^{\ba,\bb}\right|^2
+\sum_{\bn\in\Lambda_{N_2}^{2,j}}\mu_{n_j,l}\left|\hat{u}_{\bn}^{\ba,\bb}\right|^2
	:=\rom{4}_{2,1}+\rom{4}_{2,2},
\end{align*}
where 
\begin{align*}
	\Lambda_{N_2}^{1,j}:=\left\{\bn\in\Gamma_2:\ n_j>N_2^{\frac{1-\g}{d-d_1-\g}}\right\},\quad
	\Lambda_{N_2}^{2,j}:=\left\{\bn\in\Gamma_2:\ n_j\leq N_2^{\frac{1-\g}{d-d_1-\g}}\right\}.
\end{align*}
Thus,
\begin{align}\label{thm2.4_IV21}\notag
	\rom{4}_{2,1}
	\leq&\max_{\bn\in\Lambda_{N_2}^{1,j}}\left\{\frac{\mu_{n_j,l}}{\mu_{n_j,m}}\right\}\sum_{\bn\in\Lambda_{N_2}^{1,j}}\mu_{n_j,m}\left|\hat{u}_{\bn}^{\ba,\bb}\right|^2\\
	\leq&2^{l-m}|\ba|_\infty^{2(l-m)}(N_2^{\frac{1-\g}{d-d_1-\g}}-m+1)^{l-m}|u|_{\K_{\ba,\bb}^m(\Rd)}^2,
\end{align}
and by \eqref{thm2.4_eq1}, there exists some $j_0\in\{d_1+1,\cdots,d\}$ such that $n_{j_0}>N_2^{\frac{1-\g}{d-d_1-\g}}$, then
\begin{align}\label{thm2.4_IV22}\notag
	\rom{4}_{2,2}
	\leq&\max_{\bn\in\Lambda_{N_2}^{2,j}}\left\{\frac{\mu_{n_j,l}}{\mu_{n_{j_0},m}}\right\}\sum_{\bn\in\Lambda_{N_2}^{2,j}}\mu_{n_{j_0},m}\left|\hat{u}_{\bn}^{\ba,\bb}\right|^2\\
	\leq&2^{l-m}|\ba|_\infty^{2l}\left|\frac1{\ba}\right|_\infty^{2m}\left(\lfloor{N_2^{\frac{1-\g}{d-d_1-\g}}}\rfloor-m\right)^{l-m-2}|u|^2_{\K_{\ba,\bb}^m(\Rd)},
\end{align}
where $\lfloor\cdot\rfloor$ denotes the largest integer smaller or equal to $\cdot$. The estimate of $\rom{4}_2$ follows immediately from \eqref{thm2.4_IV21} and \eqref{thm2.4_IV22}:
\begin{align}\label{thm2.4_IV2}
	\rom{4}_2\lesssim|\ba|_\infty^{2(l-m)}N_2^{\frac{1-\g}{d-d_1-\g}(l-m)}|u|_{\K_{\ba,\bb}^m(\Rd)}^2.
\end{align}
The desired result follows from \eqref{thm2.4_IV1} and \eqref{thm2.4_IV2}.
\end{proof}

\end{document}